\documentclass[11pt]{article}
\usepackage{enumitem}
\usepackage[all]{xy}
\usepackage[noadjust]{cite}
\usepackage{amsfonts}
\usepackage{amsthm}
\usepackage{graphicx}
\usepackage{mathrsfs}
\usepackage{bm}
\usepackage{cite}
\usepackage[colorlinks,linkcolor=blue,citecolor=blue]{hyperref}
\usepackage{amssymb,amsmath} 
\usepackage{makecell}
\usepackage{tikz}
\usepackage{pgf}
\usepackage{tikz}
\usetikzlibrary{patterns}
\usepackage{pgffor}
\usepackage{pgfcalendar}
\usepackage{pgfpages}
\usepackage{shuffle,yfonts}
\usepackage{mathtools}
\DeclareFontFamily{U}{shuffle}{}
\DeclareFontShape{U}{shuffle}{m}{n}{ <-8>shuffle7 <8->shuffle10}{}

\newcommand{\nc}{\newcommand}
\newcommand\ta{{\texttt{x}_0}}
\newcommand\tb{{\texttt{x}_1}}
\newcommand\tc{{\texttt{x}_{-1}}}
\nc{\tz}{\tilde\zeta}
\nc{\AMZV}{\mathsf {AMZV}}
\nc{\ud}{\mathrm{d}}
\nc{\ES}{\mathsf {ES}}
\nc{\MZV}{\mathsf {MZV}}
\nc{\MtV}{\mathsf {MtV}}
\nc{\MTV}{\mathsf {MTV}}
\nc{\MSV}{\mathsf {MSV}}
\nc{\MMV}{\mathsf {MMV}}
\nc{\MMVo}{\mathsf {MMVo}}
\nc{\MMVe}{\mathsf {MMVe}}
\nc{\AMMV}{\mathsf {AMMV}}
\nc{\AMTV}{\mathsf {AMTV}}
\nc{\AMtV}{\mathsf {AMtV}}
\nc{\AMSV}{\mathsf {AMSV}}
\nc{\CMZV}{\mathsf {CMZV}}
\nc{\sha}{\shuffle}
\nc{\cst}{\rotatebox[origin=c]{180}{$\sha$}}
\nc{\cstt}{\rotatebox[origin=c]{180}{$\scriptstyle \sha$}}
\nc{\de}{\delta}
\nc{\DD}{{\mathbb D}}
\nc{\anbb}[1]{\left\langle#1\right\rangle}
\nc{\bibb}[1]{\left\{#1\right\}}
\nc{\mibb}[1]{\left[#1\right]}
\nc{\smbb}[1]{\left(#1\right)}
\nc{\doubb}[1]{\llbracket#1\rrbracket}
\nc{\dm}[1]{\left|#1\right|}

\nc{\Gbinom}[2]{\genfrac{(}{)}{0mm}{0}{#1}{#2}}
\nc{\gbinom}[2]{\genfrac{(}{)}{0mm}{1}{#1}{#2}}
\nc{\Rbinom}[2]{\genfrac{\langle}{\rangle}{0mm}{0}{#1}{#2}}
\nc{\rbinom}[2]{\genfrac{\langle}{\rangle}{0mm}{1}{#1}{#2}}
\nc{\Qbinom}[2]{\genfrac{[}{]}{0mm}{0}{#1}{#2}_q}
\nc{\qbinom}[2]{\genfrac{[}{]}{0mm}{1}{#1}{#2}_q}

\nc{\binq}[2]{\genfrac{[}{]}{0mm}{0}{#1}{#2}}
\nc{\tbnq}[2]{\genfrac{[}{]}{0mm}{1}{#1}{#2}}
\nc{\cinq}[2]{\genfrac{\{}{\}}{0mm}{0}{#1}{#2}}
\nc{\tcnq}[2]{\genfrac{\{}{\}}{0mm}{1}{#1}{#2}}

\nc{\mfrac}[2]{\genfrac{}{}{0pt}{}{#1}{#2}}
\nc{\tf}{\tfrac}
\nc{\db}{{\mathbb D}}
\nc{\pari}{{\rm par}}
\nc{\dk}{{\mathbb K}}
\nc{\ola}{\overleftarrow}
\nc{\ora}{\overrightarrow}
\nc{\lra}{\longrightarrow}
\nc{\Lra}{\Longrightarrow}
\nc\Res{{\rm Res}}
\nc\setX{{\mathsf{X}}}
\nc\fA{{\mathfrak{A}}}
\nc\evaM{{\texttt{M}}}
\nc\evaML{{\text{\em{\texttt{M}}}}}
\nc\z{{\texttt{z}}}
\nc\emz{\emph{\texttt{z}}}
\nc\tx{{\texttt{x}}}
\nc\txp{{\tx_1}} 
\nc\txn{{\tx_{-1}}} 
\nc\neo{{1}}
\nc{\yi}{{1}}
\nc\one{{-1}}
\nc\gD{{\Delta}}
\nc\eps{{\varepsilon}}
\nc{\bfMB}{{\bf MB}}
\nc{\bftB}{{\bf tB}}
\nc{\bfTB}{{\bf TB}}
\nc{\bfSB}{{\bf SB}}
\nc{\bfB}{{\bf B}}
\nc{\bfp}{{\bf p}}
\nc{\bfq}{{\bf q}}
\nc{\bfr}{{\bf r}}
\nc{\bfu}{{\bf u}}
\nc{\bfv}{{\bf v}}
\nc{\bfa}{{\bf a}}
\nc{\bfw}{{\bf w}}
\nc{\bfy}{{\bf y}}
\nc{\T}{\ddot{t}}
\nc{\bfe}{{\boldsymbol{\sl{e}}}}
\nc{\bfi}{{\boldsymbol{\sl{i}}}}
\nc{\bfj}{{\boldsymbol{\sl{j}}}}
\nc{\bfk}{{\boldsymbol{\sl{k}}}}
\nc{\bfl}{{\boldsymbol{\sl{l}}}}
\nc{\bfm}{{\boldsymbol{\sl{m}}}}
\nc{\bfn}{{\boldsymbol{\sl{n}}}}
\nc{\bfs}{{\boldsymbol{\sl{s}}}}
\nc{\bft}{{\boldsymbol{\sl{t}}}}
\nc{\bfx}{{\boldsymbol{\sl{x}}}}
\nc{\bfz}{{\boldsymbol{\sl{z}}}}
\nc\bfgs{{\boldsymbol \gs}}
\nc\bfgl{{\boldsymbol \lambda}}
\nc\bfsi{{\boldsymbol \gs}}
\nc\bfet{{\boldsymbol \eta}}
\nc\bfeta{{\boldsymbol \eta}}
\nc\bfeps{{\boldsymbol \eps}}
\nc\mmu{{\boldsymbol \mu}}
\nc\bfone{{\bf 1}}
\nc{\myone}{{1}}

 \nc{\calA}{{\mathcal A}}
 \nc{\calB}{{\mathcal B}}
 \nc{\calC}{{\mathcal C}}
 \nc{\calD}{{\mathcal D}}
 \nc{\calE}{{\mathcal E}}
 \nc{\calF}{{\mathcal F}}
 \nc{\calG}{{\mathcal G}}
 \nc{\calH}{{\mathcal H}}
 \nc{\calI}{{\mathcal I}}
 \nc{\calJ}{{\mathcal J}}
 \nc{\calK}{{\mathcal K}}
 \nc{\calL}{{\mathcal L}}
 \nc{\calM}{{\mathcal M}}
 \nc{\calN}{{\mathcal N}}
 \nc{\calO}{{\mathcal O}}
 \nc{\calP}{{\mathcal P}}
 \nc{\calQ}{{\mathcal Q}}
 \nc{\calR}{{\mathcal R}}
 \nc{\calS}{{\mathcal S}}
 \nc{\calT}{{\mathcal T}}
 \nc{\calU}{{\mathcal U}}
 \nc{\calV}{{\mathcal V}}
 \nc{\calW}{{\mathcal W}}
 \nc{\calX}{{\mathcal X}}
 \nc{\calY}{{\mathcal Y}}
 \nc{\calZ}{{\mathcal Z}}
  \nc{\cala}{{\mathcal a}}
 \nc{\calb}{{\mathcal b}}
 \nc{\calc}{{\mathcal c}}
 \nc{\cald}{{\mathcal d}}
 \nc{\cale}{{\mathcal e}}
 \nc{\calf}{{\mathcal f}}
 \nc{\calg}{{\mathcal g}}
 \nc{\calh}{{\mathcal h}}
 \nc{\cali}{{\mathcal i}}
 \nc{\calj}{{\mathcal j}}
 \nc{\calk}{{\mathcal k}}
 \nc{\call}{{\mathcal l}}
 \nc{\calm}{{\mathcal m}}
 \nc{\caln}{{\mathcal n}}
 \nc{\calo}{{\mathcal o}}
 \nc{\calp}{{\mathsf p}}
 \nc{\calq}{{\mathcal q}}
 \nc{\calr}{{\mathcal r}}
 \nc{\cals}{{\mathcal s}}
 \nc{\calt}{{\mathcal t}}
 \nc{\calu}{{\mathcal u}}
 \nc{\calv}{{\mathcal v}}
 \nc{\calw}{{\mathcal w}}
 \nc{\calx}{{\mathcal x}}
 \nc{\caly}{{\mathcal y}}
 \nc{\calz}{{\mathcal z}}
 \nc{\ot}{{\otimes}}
\usetikzlibrary{arrows,shapes,chains}
 \allowdisplaybreaks

\catcode`!=11
\let\!int\int \def\int{\displaystyle\!int}
\let\!lim\lim \def\lim{\displaystyle\!lim}
\let\!sum\sum \def\sum{\displaystyle\!sum}
\let\!sup\sup \def\sup{\displaystyle\!sup}
\let\!inf\inf \def\inf{\displaystyle\!inf}
\let\!cap\cap \def\cap{\displaystyle\!cap}
\let\!max\max \def\max{\displaystyle\!max}
\let\!min\min \def\min{\displaystyle\!min}
\let\!frac\frac \def\frac{\displaystyle\!frac}
\catcode`!=12

\nc{\gam}{{\gamma}}
\nc{\GG}{{\mathbb G}}
\nc{\PP}{{\mathbb P}}
\nc{\gG}{{\Gamma}}
\nc{\om}{{\omega}}
\nc{\vep}{{\varepsilon}}
\nc{\ga}{{\alpha}}
\nc{\gl}{{\lambda}}
\nc{\gb}{{\beta}}
\nc{\gd}{{\delta}}
\nc{\gf}{{\varphi}}
\nc{\gs}{{\sigma}}
\nc{\gk}{{\kappa}}
\nc{\gS}{\Sigma}
\let\oldsection\section
\renewcommand\section{\setcounter{equation}{0}\oldsection}

\allowdisplaybreaks

\DeclareMathOperator*{\dep}{dep}

\DeclareMathOperator{\dch}{dch}
\DeclareMathOperator{\wth}{wth}

\nc\UU{\mbox{\bfseries U}}
\nc\FF{\mbox{\bfseries \itshape F}}
\nc\h{\mbox{\bfseries \itshape h}}\nc\dd{\mbox{d}}
\nc\g{\mbox{\bfseries \itshape g}}
\nc\xx{\mbox{\bfseries \itshape x}}

\def\N{\mathbb{N}}
\def\Z{\mathbb{Z}}
\def\Q{\mathbb{Q}}

\def\reg{{\rm{reg}}}

\def\xx{\left(\frac{1-x}{1+x} \right)}

\def\ol{\overline}
\nc\divg{{\text{div}}}
\theoremstyle{plain}
\newtheorem{thm}{Theorem}[section]
\newtheorem{lem}[thm]{Lemma}

\newtheorem{cor}[thm]{Corollary}
\newtheorem{conj}[thm]{Conjecture}

\theoremstyle{definition}

\newtheorem{rem}[thm]{Remark}
\newtheorem{exa}[thm]{Example}

\nc{\cicc}[1]{{}_{{}^{ \bigcirc\hskip-1.2ex{#1}\hskip.3ex{}}}}
\nc{\cic}[1]{{}^{\bigcirc\hskip-1.15ex{\raisebox{-0.015cm}{\text{$\scriptscriptstyle #1$}}}\hskip.25ex{}}}
\nc{\ccic}[1]{{}^{\bigcirc\hskip-1.5ex{\raisebox{-0.015cm}{\text{$\scriptscriptstyle #1$}}}\hskip.25ex{}}}
\nc{\ncic}[1]{ {\bigcirc\hskip-1.6ex{\raisebox{-0.0cm}{\text{$\scriptstyle #1$}}}\hskip.25ex{}}}
\nc{\nncic}[1]{ {\bigcirc\hskip-2ex{\raisebox{-0.0cm}{\text{$\scriptstyle #1$}}}\hskip.25ex{}}}
\nc{\cci}[1]{{}_{{}^{ {\textstyle \bigcirc}\hskip-2.05ex{#1}\hskip-.35ex{}}}}
\nc{\ccicc}[1]{{}_{{}^{ {\textstyle \bigcirc}\hskip-1.55ex{#1}\hskip-0.1ex{}}}}
\nc{\x}{\rm{x}}
\nc{\tworow}[2]{\left(#1 \atop #2\right)}
\nc{\tworowT}[2]{\left(\left.{#1 \atop #2}\right|T\right)}

\nc{\fl}{{\mathfrak l}}
\nc{\fm}{{\mathfrak m}}

\begin{document}
\title{\bf Unramified Motivic Multiple Mixed Values}
\author{
{Ce Xu${}^{a,}$\thanks{Email: cexu2020@ahnu.edu.cn, ORCID 0000-0002-0059-7420.}\ \ and Jianqiang Zhao${}^{b,}$\thanks{Email: zhaoj@ihes.fr, corresponding author, ORCID 0000-0003-1407-4230.}}\\[1mm]
\small a. School of Mathematics and Statistics, Anhui Normal University, Wuhu 241002, PRC\\
\small b. Department of Mathematics, The Bishop's School, La Jolla, CA 92037, USA
}

\date{}
\maketitle
\noindent{\bf Abstract.} The multiple mixed values (MMVs) are level two variants of multiple zeta values produced by
restricting the summation indices to fixed parity patterns. One particularly interesting problem is to determine exactly when
such values are actually in level one, namely, expressible in terms of multiple zeta values. To solve this completely is
beyond our current knowledge since it calls for new ideas from transcendental number theory. However, much progress has been
made on the motivic level. Previously, using the descent theory developed by Brown et al. we have tackled this problem for a
few special classes of MMVs with regular parity patterns among the  summation indices, including Hoffman's multiple $t$-values, Kaneko and Tsumura's multiple $T$-values, and our own multiple $S$-values. In this paper, we turn to the general case and determine completely all the unramified motivic MMVs of depth less than four as well as a few families of unbounded depths. At the end of the paper, we will present some general conjectures to describe the unramified MMVs at all depths greater than three.


 
\medskip
\noindent
\textbf{Acknowledgement.} This work was started and completed while both authors visited the Tianyuan Mathematical Center
in Southeast China, supported by the funding number 12526102, during the summer of 2026. They heartily thank their hosts Prof. L. Lai and Prof. H. Zhu for the invitation.

\section{Introduction}
\subsection{Background and motivation}

In recent years, the study of multiple zeta values has developed into a vibrant area of modern mathematics, bridging number
theory, algebraic geometry, and theoretical physics. For \emph{admissible} indices $\bfs=(s_1,\ldots,s_d)\in \N^d$ with
$s_d\ge 2$, the \emph{multiple zeta values} (MZVs) are defined by
\begin{equation}\label{equ:defnMZV}
\zeta(s_1,\ldots,s_d):=\sum_{0<n_1<\cdots<n_d}\frac{1}{n_1^{s_1}\cdots n_d^{s_d}},
\end{equation}
and they exhibit a remarkably rich structure of $\Q$-linear relations that continues to captivate researchers. We refer to
$\dep(\bfs)=d$ as the depth and $|\bfs|=s_1+\cdots+s_d$ as the weight. Despite its substantial progress driven by the theory
of mixed Tate motives over $\Q$, as developed by Deligne, Goncharov, Brown, and others, the celebrated conjectures of Zagier
and Hoffman concerning the $\Q$-vector space spanned by MZVs remain open.

A key turning point has been the geometric realization of MZVs as iterated integrals on $\mathbb{P}^1\setminus\{0,1,\infty\}$,
which has proved indispensable for uncovering their motivic origin:
\begin{equation*}
\zeta(s_1,\ldots,s_d)=\int_0^1  \tb\ta^{s_1-1} \cdots  \tb\ta^{s_d-1},
\end{equation*}
where $\ta=dt/t$ and $\tb=dt/(1-t)$. This viewpoint, originally suggested by Kontsevich, interprets MZVs as periods of mixed
Tate motives. Within this motivic framework, the action of the motivic Galois group offers a powerful language for analyzing
the algebraic relations among these values. A crucial observation is that MZVs form an algebra: the product of two iterated
integrals satisfies the shuffle product formula, which greatly enriches their algebraic structure, complementing the already
apparent stuffle product structure arising directly from their series definition.

\subsection{Level two variants: multiple mixed values}\label{subsec:VarDef}

Beyond classical MZVs, a rich landscape of variants has emerged by restricting summation indices to fixed parity patterns.
These level two variants include:

\begin{enumerate}[label=\upshape{(\arabic*)},leftmargin=2cm]
    \item Hoffman's multiple $t$-values: Defined by restricting indices to odd integers \cite{Hoffman2019}:
\begin{equation*}
t(\boldsymbol{k}):=\sum_{\substack{0<n_1<\cdots<n_d\\ n_j:\text{ odd}}}\frac{2^d}{n_1^{s_1}\cdots n_d^{s_d}}.
\end{equation*}

\item Kaneko and Tsumura's multiple $T$-values \cite{KanekoTs2020}: Defined by imposing the congruence condition $n_j \equiv
    j \pmod{2}$:
\begin{equation*}
T(s_1,\ldots,s_d):=\sum_{\substack{0<n_1<\cdots<n_d\\ n_j\equiv j\pmod{2}}}\frac{2^d}{n_1^{s_1}\cdots n_d^{s_d}}.
\end{equation*}

\item Multiple $S$-values: Defined by the condition $n_j \equiv j+1 \pmod{2}$, studied previously by the current authors
    \cite{XuZhao2020a}.

\item More generally, multiple mixed values (MMVs) (see  \cite{XuYanZhao2022Aug,XuZhao2020a,XuZhao2023Aug}) allow all
    possible parity patterns, forming a subspace of the $\Q$-space of alternating multiple zeta values (Euler sums). These
    MMVs encompass the above variants as special cases and possess elegant algebraic structures including stuffle, shuffle,
    and duality relations analogous to those of ordinary MZVs.
\end{enumerate}
All the above values can be expressed as $\Z$-linear combinations of Euler sums which are defined by
\begin{equation*}
\zeta(s_1,\ldots,s_d;\eps_1,\dots,\eps_d):=\sum_{0<n_1<\cdots<n_d}\frac{\eps_1^{n_1}\cdots \eps_d^{n_d}}{n_1^{s_1}\cdots
n_d^{s_d}}
\end{equation*}
for $(s_1,\ldots,s_d)\in {\N}^d, (\eps_1,\dots,\eps_d)\in\{\pm1\}^d$, with $(s_d,\eps_d)\ne(1,1).$ We often suppress the
$\eps_j$'s with a short-hand notation by putting a bar on top of $s_j$ if $\eps_j=-1$. For example, $\zeta(\bar 1)=\zeta(1;
-1)$ is the alternating harmonic series. Like MZVs, these Euler sums also have two product structures, one from their iterated
integral expressions (leading to their motivic version)
\begin{equation*}
\zeta(s_1,\ldots,s_d;\eps_1,\dots,\eps_d)=\int_0^1 \tx_{\eps_1\dots\eps_d} \ta^{s_1-1} \tx_{\eps_2\dots\eps_d}
\ta^{s_2-1}\cdots \tx_{\eps_d} \ta^{s_d-1}
\end{equation*}
where $\tx_\eta=dt/(\eta-t)$,
and the other from the stuffle product of their series expressions. For example, for $a,b\in \N\cup \ol{\N}$ where the set $\ol{\N}:=\{\ol{n}:n\in\N\}$, we have the stuffle product
\begin{equation*}
  \zeta(a)\zeta(b)= \zeta(a,b)+ \zeta(b,a)+ \zeta(a\oplus b),
\end{equation*}
where for any two positive integers $m, n$ we set $m\oplus n=\ol{m}\oplus \ol{n}:=m+n$ and $\ol{m}\oplus n=m\oplus
\ol{n}:=\ol{m+n}$.

\subsection{The unramifiedness problem}

A central question in this area concerns the relationship between level two variants and classical MZVs. Following
Broadhurst's terminology, certain Euler sums that lie in the $\Q$-span of MZVs are called ``honorary MZVs''. Motivated by the
Galois descent theory developed by Glanois in her thesis \cite{Glanois2015}, such values (resp.\ their motivic counterparts) are more
systematically referred to as ``unramified'': they are expressible as $\Q$-linear combinations of MZVs (resp. motivic MZVs).

The problem of determining which MMVs are unramified, i.e.,  those which descend to level one, presents a significant
challenge. At the analytic level, this requires proving explicit identities relating Euler sums to MZVs. For example, the
beautiful family discovered numerically by Borwein et al. \cite{BorweinBrBr1997} and proved by the second author
\cite{Zhao2010a},
$$8^\ell \zeta(\{1,\bar{2}\}_\ell)=\zeta(3_\ell) \quad \text{for all } \ell\in\N,$$
exemplifies the kind of deep arithmetic structure at play. Here, for any string $\bfs$ we denote by $\{\bfs\}_\ell$ the
concatenation of $\bfs$ by itself exactly $\ell$ times, and we may drop the curly brackets for a singleton. Another result
which will be used in this paper is of two-one formula type (by taking $l=1$, $\bfs=(n)$ in \cite[Theorem~1.4]{Zhao2013a}): for
all positive integer $n\ge 2$
\begin{align}\label{equ:1nBar2}
\zeta^\sharp(\{1\}_{n-2}, \ol{2})=-\zeta^\star(n)=-\zeta(n).
\end{align}
Here $\zeta^\star$ is defined similarly as in \eqref{equ:defnMZV} except that one allows equalities among the summation
indices, and
for all  $s_1,\dots,s_d\in \N\cup \ol{\N}$, the Euler $\sharp$-sums
\begin{equation*}
\zeta^\sharp(s_1,\dots,s_d)= \sum_{\bft= (s_1\circ \dots \circ s_d) } 2^{\dep(\bft)} \zeta(\bft),
\end{equation*}
where $\circ=$``,'' or ``$\oplus$''. For example, $\zeta^\sharp(\ol{2},3)=2\zeta(\ol{5})+4\zeta(\ol{2},3)$. It is clear that
the Euler $\sharp$-sums are essentially the interpolated Euler sums at $t=1/2$ (the MZV version was first considered by
Yamamoto \cite{Yamamoto2012b}), namely, $\zeta^\sharp(\bfs)=2^{\dep(\bfs)}\zeta^{1/2}(\bfs).$

The complete resolution of the unramified problem at the analytic level appears beyond current knowledge, as it would require
new ideas from transcendental number theory, however, substantial progress has been made at the motivic level.

\subsection{Motivic approach and descent theory}

The motivic framework for level two variants builds upon Deligne--Goncharov's theory of mixed Tate motives which was further developed by Brown (see
\cite{Brown2012,Deligne2010,DeligneGo2005}). For Euler sums and their parity-restricted variants, one can define motivic versions
denoted by $\zeta^{\mathfrak{m}}$ that serve as lifts of the analytic values to a suitable motivic Hopf algebra.

The key tool in determining unramifiedness is the descent criterion developed by Brown \cite{Brown2012} and Glanois
\cite{Glanois2015,Glanois2016}. This criterion provides a necessary and sufficient condition for a motivic Euler sum to be
unramified: essentially, the motivic coaction must preserve a certain stable subspace corresponding to MZVs. This reduces the
problem to verifying a finite-dimensional linear algebra condition on the motivic coproduct.

Previous work has applied this machinery successfully to special classes of MMVs. For Hoffman's multiple $t$-values, Murakami
\cite{Murakami2021} and Charlton \cite{Charlton2021} have identified several families of unramified MtVs, including all MtVs
with components at least 2. More such families can be found in the authors' related work \cite{XuZhao2023Aug}.  There, the authors have derived necessary and sufficient conditions for unramifiedness of the motivic MSVs and MTVs when the depth is less than four, partially confirming a conjecture of Kaneko and Tsumura.
Tsumura.

\subsection{Main results}

In this paper, we turn to the general setting of multiple mixed values (MMVs). Our goal is three-fold: first, to determine
completely all unramified motivic MMVs in depth less than four; second, to identify several new families of unramified MMVs of
unbounded depth; third, to present some general conjectures on all possible unramified motivic MMVs in large depths.

Specifically, building on the descent theory established in prior works, we provide:

\begin{enumerate}[label=\upshape{(\arabic*)},leftmargin=2cm]
    \item
 Complete classification for motivic MMVs of depth $d \leq 3$ that are unramified, giving necessary and sufficient
 conditions in terms of the index parity patterns (see Theorem~\ref{thm:triUnramMMV}).

\item  Two infinite families of unramified motivic MMVs of almost extremal height, extending known results for MMVs of
    regular parity patterns (see Theorem~\ref{thm:UnramMMVfamily1} and \ref{thm:UnramMMVfamily2}).

\item  Explicit identities (conditional on an analytic conjecture in one case) expressing each motivic value in a height one
    family, with $d$ members in every even depth $d$, as a $\Q$-linear combination of motivic MZVs (see
    Theorem~\ref{thm:UnramMMVfamily3} and \ref{thm:assumConjMotivic}).

\item  Conjectures \ref{conj:unramMTVdual}-\ref{conj:depth>5MMVs} on all possible unramified motivic MMVs in arbitrary
    depths, based on our extensive numerical experiment.
\end{enumerate}

These results represent a significant step toward understanding the general MMV landscape. The classification theorems
establish, under Grothendieck's period conjecture, complete characterizations of unramifiedness for low-depth cases,
confirming and extending conjectures of Kaneko and Tsumura and others.

\subsection{Outline of the Paper}

The remainder of this paper is organized as follows. In Section~\ref{sec:setup}, we establish the necessary definitions and notation for
MMVs and their motivic counterparts after a quick review of the set-up to apply the Galois descent, including the criterion of Brown and Glanois on
determining when Euler sums are unramified. Section~\ref{sec:lemmas} is a brief
summary of a few important previous results we will need throughout the paper. Section~\ref{sec:UnramDepth3} presents the classification results for unramified MMVs of depth three, with detailed proofs for the various parity patterns. Section~\ref{sec:RamDepth3} will be devoted to show that motivic MMVs of depth
three not covered in Section~\ref{sec:UnramDepth3} are all ramified, namely, there are only three unramified motivic MMVs having irregular parity patterns. This section is the most technical part of the paper and we will present a complete proof for one pattern and then
provide all the necessary key steps for the other three patterns. Section~\ref{sec:2family} constructs two infinite unramified families of
arbitrary depth $\ge 2$ and almost extremal height while Section~\ref{sec:htOneFamily} provides one more unramified family but with height equal to one. Moreover, in every even depth $d$ we determine $d$ such unramified height one motivic MMVs with their explicit
expressions in terms of motivic MZVs with one exception (where the expression is only conjectural). Finally, Section~\ref{sec:Conj}
discusses open problems and conjectures suggested by our results and extensive numerical experiments, including potential
generalizations and connections to the ramification theory of more general variations of MZVs such as the alternating MMVs
\cite{XuYanZhao2022Aug} and multiple Clausen values and multiple Glaishers values at level three \cite{BorweinBrKa2001}.

\section{Notation and the motivic setup}\label{sec:setup}
\subsection{Multiple mixed values}
We first remark that all the variants of MZVs in Subsection~\ref{subsec:VarDef}(1)-(4) are level two cases of the objects
considered by Yuan and the second author in \cite[(2.1)]{YuanZh2014a} where they defined more generally the (congruence) MZVs
of level $N$ by restricting the summation indices to a fixed sequence of congruence classes modulo $N$. At level two, for any
admissible indices of positive integers $\bfs=(s_1,\ldots,s_d)$ (admissible means $s_d\ge 2$) and $\bfeps=(\eps_1, \dots,
\eps_d)\in\{\pm 1\}^d$, we define in \cite{XuZhao2020a} the \emph{multiple mixed values} or \emph{multiple $M$-values} (MMVs)
by
\begin{equation}\label{equ:MMVdefn}
M(\bfs;\bfeps):=\sum_{0<m_1<\cdots<m_d} \frac{(1+\eps_1(-1)^{m_1}) \cdots (1+\eps_d(-1)^{m_d})}{m_1^{s_1} \cdots
m_d^{s_d}}=\sum_{\substack{0<n_1<\cdots<n_d\\ 2| n_j \text{ if } \eps_j=1 \\ 2\nmid n_j \text{ if } \eps_j=-1
}} \frac{2^r}{n_1^{s_1} \cdots n_d^{s_d}}.
\end{equation}
As usual, we call $|\bfs|:=s_1+\cdots+s_d$ and $d$ the \emph{weight} and \emph{depth}, respectively.
For convenience, we say the \emph{(parity) signature} of $s_j$ is \emph{even} or \emph{odd} depending on whether $\eps_j$ is 1
or $-1$. For brevity, we put a check on top of the component $s_j$ if $\eps_j=-1$. For example,
\begin{align*}
M(2,\check{3})=\sum_{0<n_1<n_2} \frac{(1+(-1)^{n_1}) (1-(-1)^{n_2})}{n_1^2 n_2^3}
= \sum_{0<m<n} \frac{4}{(2m)^2 (2n-1)^3}.
\end{align*}

Further, every MMV has an iterated integral expression as follows. Setting $\eps_0=1$,
\begin{equation}\label{defn:oms}
\om_0:=\tx_0=\frac{dt}t,\quad \om_{-1}(t):=\frac{2dt}{1-t^2}=\tb-\tc, \quad\om_1(t):=\frac{2tdt}{1-t^2}=\tb+\tc,
\end{equation}
where $\tx_\eta=dt/(\eta-t)$ for $\eta=\pm1$ as before, then
\begin{align}
M(\bfs;\bfeps)=&\, \int_0^1 \om_{\eps_1}\om_0^{s_1-1}
\om_{\eps_1\eps_2}\om_0^{s_2-1} \cdots \om_{\eps_{d-1}\eps_d} \om_0^{s_d-1}, \notag\\
=&\, \sum_{\eta_1,\dots,\eta_d=\pm 1} \bigg(\prod_{j=1}^d \eta_j^{\scriptstyle \frac{\eps_j-\eps_{j-1}}2} \bigg)\int_0^1
\tx_{\eta_1}\ta^{s_1-1}
\tx_{\eta_2}\ta^{s_2-1} \cdots \tx_{\eta_d}\om_0^{s_d-1}. \label{equ:MMVint}
\end{align}

\subsection{Motivic setup}

Motivic setup developed in \cite{Brown2012,Glanois2016} will be enforced throughout this paper. As we will be concerned only
with level 2, we will recall briefly the main framework in this case. Let $w\in\N$ and suppose $a_j\in\{0,\pm 1\}$ for all $0\le
j\le w+1$. Then one can define the motivic integrals $I^\fm(a_0;a_1,\dots,a_w;a_{w+1})$ as symbols satisfying a list of
axioms. The key property of these motivic integrals is that there exists a period map $\dch$ defined by
\begin{equation}\label{equ:periodMap}
\dch \big(I^\fm(a_0;a_1,\dots,a_w;a_{w+1})\big):=(-1)^{\dep(a_1,\dots,a_w)}
\int_{a_0}^{a_{w+1}} \frac{dt}{t-a_1} \cdots \frac{dt}{t-a_w}
\end{equation}
as an iterated integral whenever it converges (to some Euler sums, see \cite[Ch.\ 14]{Zhao2016}).
Here $\dep(a_1,\dots,a_w)$ is the number of nonzero components in $(a_1,\dots,a_w)$. This should be modified if the integral
diverges, as we will see below.

Let $\calH^2$ be the $\Q$-vector space spanned by the \emph{motivic Euler sums} of the form
\begin{equation*}
\zeta^\fm_a\tworow{s_1,\dots,s_d}{\eps_1,\dots,\eps_d}=\zeta^\fm_a(s_1,\dots,s_d;\eps_1,\dots,\eps_d):=
I^\fm(0;0_a,\eta_1,0_{s_1-1},\dots,\eta_d,0_{s_d-1};1),
\end{equation*}
where $a\in\N_0:=\N\cup\{0\}$, $s_j\in\N, \eps_j=\pm 1$, $\eta_j=\eps_j\cdots \eps_d$ for all $j=1,\dots,d$.
We simply write $\zeta^\fm$ for $\zeta^\fm_0$.
These are the motivic MZVs when $\eps_j=1$ for all $j$, which together span the level one motivic MZV space denoted by
$\calH^1$. The motivic integrals satisfy a list axiom which can be found in \cite[\S2.4]{Brown2012}, \cite[p.~9]{Glanois2016}
or \cite[\S2, (I1)-(I6)]{Murakami2021}). We will omit these to avoid repetitiveness. In this setting, we have a more concrete
description of the period map $\dch$: if $f\in\calH_n^2$ with a diverge image under \eqref{equ:periodMap} but finite
shuffle-regularized value $f^\reg$, then $\dch(f)=f^\reg$. For example, the shuffle-regularized value of a non-admissible
Euler sum
\begin{equation*}
\zeta^\reg\tworow{\ldots,1}{\ldots,1}:={\L}_\sha(\ldots,1|\ldots,1;T)|_{T=0}
\end{equation*}
where ${\L}_\sha(\ldots,1|\ldots,1;T)$ is defined in \cite[Prop.~13.3.8]{Zhao2016}.

Let $N=1$ or 2. Set $\calA^N=\calH^N/\zeta^\fm(2)\calH^N$. Denote by
$\calH^N_w$ and $\calA_w^N$ their weight $w$ part for all $w\ge 0$. Let $\calL^N=\calA^N_{>0}/\calA^N_{>0}\cdot \calA^N_{>0}$.
For any weight $w$ and odd $r$ such that $0<r<w$ by modifying the coproduct of certain Hopf algebra
one can define a derivation as part of a coaction
$$
D_r: \calH_w^N \to \calL_r^N \ot \calH_{w-r}^N
$$
by sending $I^\fm(a_0;a_1,\dots,a_w;a_{w+1})$ to
$$
\sum_{p=0}^{w-r} I^\fl(a_p; a_{p+1},\dots,a_{p+r};a_{p+r+1})\ot I^\fm(a_0;a_1,\dots,a_p,a_{p+r+1},\dots,a_w;a_{w+1}).
$$
Each summand provides a choice of a consecutive subsequence from $(a_0,\dots,a_{w+1})$ and is therefore called a \emph{cut}.

The following theorem plays the pivoting role in the descent theory of Euler sums. It contains both Glanois's result \cite[Cor.\
2.4]{Glanois2016} and Brown's \cite[Thm.\ 3.3]{Brown2012}. Set $\calH=\calH^1$ and $\calL=\calL^1$.

\begin{thm}\label{thm:Glanois}
Let $a\in\N_0$, $\bfeps\in\{\pm 1\}^d$, and $\bfs\in\N^d$ such that $a+|\bfs|=w$.
Then the weight $w$ motivic Euler sum $\zeta_a^\fm\tworow{\bfs}{\bfeps}\in\calH_w$ if and only $D_1
\zeta_a^\fm\tworow{\bfs}{\bfeps}=0$ and
$D_r \zeta_a^\fm\tworow{\bfs}{\bfeps}\in \calL_r\ot \calH_{w-r}$ for all odd $r<w$.
Moreover,  $  \bigcap_{r<w} \ker ( D_r:\calH_w^2 \to \calL_r^2 \ot \calH_{w-r}^2) =\zeta^\fm(n)\Q$.
\end{thm}

For all $a\in\N_0$ and $\bfs=(s_1,\dots,s_d)\in\N^d$, motivated by \eqref{equ:MMVint} we define the motivic MMVs by
\begin{equation}
M^\fm(\bfs;\bfeps)=  \sum_{\eta_1,\dots,\eta_d=\pm 1} \bigg(\prod_{j=1}^d \eta_j^{\scriptstyle \frac{\eps_j-\eps_{j-1}}2}
\bigg)
I^\fm(0;\eta_1,0^{s_1-1},\eta_2,0^{s_2-1}, \ldots,\eta_d,0^{s_d-1};1). \label{equ:defnMotivicMMV}
\end{equation}
Then, $\dch M^\fm(\bfs)=M(\bfs)$ for all admissible $\bfs$. Moreover, it follows easily that the  motivic multiple $t$-, $T$-,
$S$- and $\tz$-values are given by \cite[(2.2)-(2.5)]{XuZhao2023Aug}, respectively. We say these MMVs have \emph{regular
parity patterns}. Here,
$\tz^\fm(\bfs)=2^{\dep(\bfs)-|\bfs|} \zeta^\fm(\bfs)$ is related to the MZV in \eqref{equ:defnMZV} when we only allow even
indices to appear.
We can further define, for any $a\in\N_0$,
\begin{align}\label{defn:MMtV}
M^\fm_a(\bfs)=&\, \sum_{\eta_1,\dots,\eta_d=\pm 1} \bigg(\prod_{j=1}^d \eta_j^{\scriptstyle \frac{\eps_j-\eps_{j-1}}2} \bigg)
I^\fm(0;0_a,\eta_1,0^{s_1-1},\eta_2,0^{s_2-1}, \ldots,\eta_d,0^{s_d-1};1).
\end{align}

\section{Some technical lemmas}\label{sec:lemmas}

We will need the following results. The proof of Lemma~\ref{lem:D1S1bfl} below can be found in the proof of \cite[Lemma~ 3.2,
Lemma~3.3, (5.2), (5.3), (5.7)]{XuZhao2023Aug}, \cite[Theorem~1 and Lemma~9]{Murakami2021} and \cite[Prop.~5.9]{Charlton2021}.
For convenience, we define the Kronecker symbol $\delta_\bullet=1$ if the condition $\bullet$ is true and $\delta_\bullet=0$
otherwise.

\begin{lem}\label{lem:D1S1bfl}
Assume $\bfs=(s_1,\dots,s_d)\in\N^d$ and $k\in\N$. Then
 \begin{enumerate}[label=\upshape{(\arabic*)},leftmargin=1cm]
 \item \label{lemCase:T(1)etc} $T^\fm(1)=t^\fm(1)=\log^\fm 2$, $\tz^\fm(1)=S^\fm(1)=-\log^\fm 2$.
 \item \label{lemCase:T=tz,S=tz} $T^\fm(k)=t^\fm(k)=(2^k-1)\tz^\fm(k)$, $S^\fm(k)=\tz^\fm(k)$ if $k\ge2$.
 \item \label{lemCase:D1MTVadm} $D_1 T^\fm(\bfs)=0$ if $\bfs$ is admissible.
 \item \label{lemCase:D1MTV(k,1)} $D_1 T^\fm(k,1)=\log^\fl 2 \ot T^\fm(k)\ne 0$.
 \item \label{lemCase:D1MSV(k,1)} $D_1 S^\fm(k,1)=\log^\fl 2 \ot S^\fm(k)\ne 0$.
 \item \label{lemCase:D1MSV} $D_1 S^\fm(\bfs)=-2\delta_{s_1=1} \log^\fl 2\ot T^\fm(s_2,\dots,s_d)$ if $s_d\ne1$.
 \item \label{lemCase:D1MtV} $D_1 t^\fm(\bfs)=  \log^\fl \ot \Big(2 \delta_{s_1=1}t^\fm(s_2,\dots,s_d) -\delta_{s_d=1}
     t^\fm(s_1,\dots,s_{d-1}) \Big)$.
 \item \label{lem:MtVUnram} If all components of $\bfs$ are at least $2$ then $t^\fm(\bfs)$ is unramified.
 \item \label{lemCase:DTDS} If $a+b>0$ is even, then
\begin{align*}
D_{a+b+1} S^\fm(a+1,b+1)=&\, \bigg[\delta_{b=0}\tz^\fl(a+b+1)-(-1)^a 2\binom{a+b}{a}T^\fl(a+b+1)\bigg]\ot \log^\fm 2,\\
D_{a+b+1} T^\fm(a+1,b+1)=&\, \bigg[(-1)^a 2\binom{a+b}{a}-\delta_{b=0}\bigg]T^\fl(a+b+1)\ot \log^\fm 2,\\
D_{a+b+1} t^\fm(a+1,b+1)=&\, \gd_{b=0}T^\fl(a+b+1)  \ot \log^\fm 2.
\end{align*}
In particular, all motivic double $S$ and $T$-values of even weights are ramified.
 \item \label{Murakami-lemma9} Let $s_0,s_1,\dots,s_d\in\N_0$ and $\ga,\gb=\pm 1$. Then
\begin{equation*}
    \sum_{\eta_j=\pm 1} I^\fl(\ga;0_{s_0}, \eta_1,0_{s_1}, \dots, \eta_d,0_{s_d};\gb)=0.
\end{equation*}
\end{enumerate}
\end{lem}

In Lemma~\ref{lem:D1S1bfl}\ref{Murakami-lemma9} and throughout the rest of this paper, to save space we will use the notation
\begin{equation*}
    \sum_{\eta_j= \pm 1}  (  \cdots  ) =  \sum_{(\eta_1,\ldots,\eta_d)\in \{\pm 1\}^d}  (  \cdots  ).
\end{equation*}

\begin{lem}\label{lem:DBMMVs} \emph{(\cite[Cor. 4.2]{XuZhao2023Aug})}
Suppose $m,n\in\N$ and $w =m + n$ is odd. Then
\begin{align*}
t^\fl(m,n)=-t^\fl(w),\quad
T^\fl(m,n)= (-1)^{n} \binom{w-1}{n} T^\fl(w), \quad
S^\fl(m,n)= (-1)^{n} \binom{w-1}{m} T^\fl(w).
\end{align*}
\end{lem}

\begin{lem}\label{lem:S_a(b+1,c+1)}
Let $a\in\N_0$, $m,n\in\N$  and let $w=a+m+n$ be odd. Then
\begin{align*}
S_a^\fl(m,n)= &\, (-1)^{a+m-1} \binom{w-1}{n-1} T^\fl(w),\\
T_a^\fl(m,n)= &\, (-1)^{a+m-1} \binom{w-1}{m-1} T^\fl(w),\\
t_a^\fl(m,n)= &\, -(-1)^{a} \binom{w-1}{a} T^\fl(w).
\end{align*}
\end{lem}
\begin{proof}
Let $m=b+1$ and $n=c+1$. By definition and Lemma~\ref{lem:DBMMVs},
\begin{align*}
S_a^\fl(b+1,c+1)= &\, (-1)^a \sum_{i+j=a} \binom{b+i}{i}\binom{c+j}{j} S^\fl(b+i+1,c+j+1)\\
= &\, \sum_{i+j=a} (-1)^{a+b+i} \binom{b+i}{i}\binom{c+j}{j}\binom{w-1}{b+i+1} T^\fl(w)\\
= &\, \left. \frac{(-1)^{a+b} (w-1)!}{a!b!c!}  T^\fl(w) \sum_{i+j=a} (-1)^{i} \binom{a}{i}\frac{x^{b+i+1}}{b+i+1}
\right|_{x=1}\\
= &\,  \frac{(-1)^{a+b} (w-1)!}{a!b!c!}  T^\fl(w) \int_0^1 x^b (1-x)^a \, dx\\
= &\,  \frac{(-1)^{a+b} (w-1)!}{(a+b+1)! c!}  T^\fl(w)
\end{align*}
by the Euler beta integral. The formula for $T_a^\fl(b+1,c+1)$  follows immediately from the fact that
$T^\fl(m,n)=-S^\fl(n,m)$.

Turing to the $t$-values, from Lemma~\ref{lem:DBMMVs}  we have
\begin{align*}
t^\fl_a(b+1,c+1)=-(-1)^{a} \sum_{i+j=a} \binom{b+i}{i}\binom{c+j}{j} T^\fl(w)
=-(-1)^{a} \binom{w-1}{a} T^\fl(w)
\end{align*}
by the generalized Vandermonde identity.
This completes the proof of the lemma.
\end{proof}

\section{Unramified MMVs in depth three}\label{sec:UnramDepth3}
We will classify all unramified MMVs in depth three in this section. After dealing with the three special cases we will describe the general situation in Theorem~\ref{thm:triUnramMMV}. Then we will present the main body of the proof for the irregular parity cases in Section~\ref{sec:RamDepth3}.

\subsection{Three special unramified MMVs in depth three}\label{subsec:SpecialUnramDepth3}
In this subsection, we will prove the unramifiedness of three special MMVs. Later, we will see that these are the only unramified
MMVs with irregular parity patterns in depth three.

\begin{lem}\label{lem:3SpecialMMVs}
The motivic MMVs (i) $M^\fm(\check{1},\check{1},3)$, (ii) $M^\fm(\check{1},2,2)$ and (iii) $M^\fm(\check{1},4,4)$ are all
unramified.
\end{lem}
\begin{proof}
(i) For $M^\fm(\check{1},\check{1},3)$, we have
\begin{align*}
D_1 &\, M^\fm(\check{1},\check{1},3)= \sum_{\eta_j=\pm 1} \eta_1\eta_3 D_1 I^\fm(0;\eta_1,\eta_2,\eta_3,0,0;1)\\
=&\, \sum_{\eta_j=\pm 1} \eta_1\eta_3   \Big[I^\fl(0;\eta_1;\eta_2)\ot I^\fm(0;\eta_2,\eta_3,0,0;1)
+ I^\fl(\eta_1;\eta_2;\eta_3)\ot I^\fm(0;\eta_1,\eta_3,0,0;1) \\
&\, \hskip5cm + I^\fl(\eta_2;\eta_3;0)\ot I^\fm(0;\eta_1,\eta_2,0,0;1) \Big]\\
=&\, \sum_{\eta_j=\pm 1} \eta_1\eta_2\eta_3   \Big[I^\fl(0;\eta_1;1)\ot I^\fm(0;\eta_2,\eta_3,0,0;1)
+ I^\fl(1;\eta_1;0)\ot I^\fm(0;\eta_3,\eta_2,0,0;1) \Big]=0.
\end{align*}
Here, the second term in the second line disappears by Lemma~\ref{lem:D1S1bfl}\ref{Murakami-lemma9}. Further,
\begin{align*}
D_3 &\,M^\fm(\check{1},\check{1},3)=  \sum_{\eta_j=\pm 1} \eta_1\eta_3 D_3 I^\fm(0;\eta_1,\eta_2,\eta_3,0,0;1)\\
=&\, \sum_{\eta_j=\pm 1} \eta_1\eta_3  \Big[I^\fl(\eta_1;\eta_2,\eta_3,0;0)\ot I^\fm(0;\eta_1,0;1)
+ I^\fl(\eta_2;\eta_3,0,0;1)\ot I^\fm(0;\eta_1,\eta_2;1) \Big]\\
=&\, \sum_{\eta_j=\pm 1}  \Big[ \eta_3 I^\fl(1;\eta_2,\eta_3,0;0)\ot I^\fm(0;\eta_1,0;1)
+ \eta_1\eta_3 I^\fl(\eta_2;\eta_3,0,0;0)\ot I^\fm(0;\eta_1,\eta_2;1) \\
&\,  \hskip5cm + \eta_1\eta_3 I^\fl(0;\eta_3,0,0;1)\ot I^\fm(0;\eta_1,\eta_2;1)  \Big]\\
=&\, \sum_{\eta_j=\pm 1}  \Big[ -t_1^\fl(1,1)\ot \tz^\fm (2)
+ \eta_3 I^\fl(1;\eta_3,0,0;0)\ot \eta_1\eta_2 I^\fm(0;\eta_1,\eta_2;1)
  + T^\fl(3)\ot t^\fm(1,1)  \Big]\\
=&\, \sum_{\eta_j=\pm 1}  \Big[ \Big(t^\fl(2,1)+t^\fl(1,2)\Big)\ot  \tz^\fm(2)
-T_2^\fl(1)\ot T^\fm(1,1)+ T^\fl(3)\ot t^\fm(1,1)   \Big]\\
=&\, - 2T^\fl(3)\ot \tz^\fm(2)+T^\fl(3)\ot\Big[t^\fm(1,1) -T^\fm(1,1)\Big]
\end{align*}
by Lemma~\ref{lem:DBMMVs}. Since $D_1 (t^\fm(1,1) -T^\fm(1,1))=0$ by Lemma~\ref{lem:D1S1bfl}\ref{lemCase:D1MTV(k,1)} and
\ref{lemCase:D1MtV} we see that $D_3M^\fm(\check{1},\check{1},3)\in \calL_3\ot \calH_2$ by Theorem~\ref{thm:Glanois}. Together
with the vanishing under $D_1$, this shows that $M^\fm(\check{1},\check{1},3)\in \calH_5$ is  unramified by
Theorem~\ref{thm:Glanois}.

(ii) For $M^\fm(\check{1},2,2)$, we can mimic case (i) to get
\begin{align*}
D_1M^\fm(\check{1},2,2)=&0,\\
D_3M^\fm(\check{1},2,2)=&\, 2 T^\fl(3)\ot \tz^\fm(2)\in \calL_3\ot \calH_2.
\end{align*}
This shows that $M^\fm(\check{1},2,2)\in \calH_5$ is unramified by Theorem~\ref{thm:Glanois}.

(iii) Similarly, for $M^\fm(\check{1},4,4)$, we have
\begin{align*}
D_1M^\fm(\check{1},4,4)=&\, 0,\\
D_3M^\fm(\check{1},4,4)=&\, -t^\fl(3)\ot \tz^\fm(2,4)\in\calL_3\ot\calH_6,\\
D_5M^\fm(\check{1},4,4)=&\,2 T^\fl(5)\ot \tz^\fm(4)\in\calL_5\ot\calH_4,\\
D_7M^\fm(\check{1},4,4)=&\, T^\fl(7)\ot \tz^\fm(2)\in\calL_7\ot\calH_2,
\end{align*}
and therefore $M^\fm(\check{1},4,4)$ is unramified by Theorem~\ref{thm:Glanois}.
\end{proof}

\begin{rem}
Using MZV data mine \cite{BlumleinBrVe2010}, we can quickly find that
\begin{align}
M(\check{1},\check{1},3)=&\,\frac{1}{2^5}\Big(217\zeta(5)-112 \zeta(2)\zeta(3) \Big)
= \frac{217}{2} \tz(5)-2 T(3)\zeta(2), \label{equ:analyticM113}\\
M(\check{1},2,2)=&\,\frac{1}{2^5}\Big(-93\zeta(5)+56\zeta(2)\zeta(3) \Big)=-\frac{93}{2}\tz(5)+T(3) \zeta(2) ,\notag\\
M(\check{1},4,4) =&\,\frac{1}{2^9}\Big(-\frac{10585}9\zeta(9)+508\zeta(7)\zeta(2)+248\zeta(5)\zeta(4)
    +56\zeta(3)\zeta(6)-\frac{112}{3}\zeta(3)\zeta(3,3)\Big) \notag \\
=&\,-\frac{10585}{18} \tz(9)+  T(7)\tz(2)+2 T(5)\tz(4)+2 T(3)\tz(6)-\frac{2}{3}T(3)\tz(3,3)  .\notag
\end{align}
All these relations can be lifted to the motivic version by checking that the images under all coactions $D_r$ (for odd $r$)
are the same on the two sides. For example, by duality relation
\begin{align*}
t^\reg(1,1) -T^\reg(1,1)=&\,\int_0^1 \om_{-1}(\om_1-\om_{-1}) \\
=&\,\int_0^1(\om_1-\om_{-1})\om_0=S(2)-T(2)=2\zeta(\ol{2})=-\zeta(2).
\end{align*}
This implies that $D_3 M^\fm(\check{1},\check{1},3)=-2 T^\fl(3)\ot \zeta^\fm(2)$. Since $D_3 M^\fm(\check{1},\check{1},3)=0$,
we now can lift \eqref{equ:analyticM113} to its motivic version by Theorem~\ref{thm:Glanois}.
\end{rem}

\subsection{Classification of unramified MMVs in depth three}\label{subsec:UnramDepth3}
We now summarize all the possible unramified MMVs in depth three in the following theorem.

\begin{thm}\label{thm:triUnramMMV}
Suppose $\bfs\in\N^3$ and $\bfgs\in\{\pm1\}$. Let $O$ (resp.~$E$) represent odd (resp.~even) positive integers.
Then the following list provides a complete classification
of unramified motivic triple mixed values $M^\fm(\bfs;\bfgs)$:
\begin{enumerate}
  \item[\upshape{(1)}] all $\tz(\bfs)$, i.e., $\bfgs=(1,1,1)$;
  \item[\upshape{(2)}] for $t^\fm(\bfs)$, i.e., $\bfgs=(-1,-1,-1)$:
    \begin{enumerate}
        \item [\upshape{(i)}]  $\bfs\in (\N_{\ge 2})^3$, or
        \item [\upshape{(ii)}] $\bfs=(E,1,\N_{\ge 2})$;
    \end{enumerate}

  \item[\upshape{(3)}] for $T^\fm(\bfs)$, i.e., $\bfgs=(-1,1,-1)$:
    \begin{enumerate}
 \item [\upshape{(i)}] $\bfs=(1,1,2), (1,1,3), (1,2,2)$, or
 \item [\upshape{(ii)}]  $\bfs=(O_{\ge 3},E,O_{\ge 3})$, or
 \item [\upshape{(iii)}] $\bfs=(E,1,E)$;
    \end{enumerate}

  \item[\upshape{(4)}] for $S^\fm(\bfs)$, i.e., $\bfgs=(1,-1,1)$:
  \begin{enumerate}
 \item [\upshape{(i)}]  $\bfs=(O_{\ge 3},O,E)$, or
 \item [\upshape{(ii)}] $\bfs=(O_{\ge 3},E,O_{\ge 3})$, or
 \item [\upshape{(iii)}]  $\bfs=(E,1,E)$;
\end{enumerate}

  \item[\upshape{(5)}] for other triple mixed values with irregular parity patterns:
    \begin{enumerate}
 \item [\upshape{(i)}] $(\bfs;\bfgs)=(1, 1, 3;-1, -1, 1)$, or
 \item [\upshape{(ii)}] $(\bfs;\bfgs)=(1, 2, 2;-1, 1, 1)$, or
 \item [\upshape{(iii)}] $(\bfs;\bfgs)=(1, 4, 4;-1, 1, 1)$.
    \end{enumerate}
\end{enumerate}
\end{thm}
\begin{proof}

Clearly, all values in (1) are unramified.  Turning to (2), we see that (i) is just Lemma~
\ref{lem:D1S1bfl}\ref{lem:MtVUnram}. For Part(ii), let $\bfs=(s_1,s_2,s_3)$. If $s_1=1$ or $s_3=1$, then $t^\fm(\bfs)$ is
ramified by Lemma~\ref{lem:D1S1bfl}\ref{lemCase:D1MtV}. Additionally, $t^\fm(O,1,s_3)$ is ramified by \cite[Thm.
9.13]{XuZhao2023Aug}. On the other hand, $t^\fm(E,1,\N_{\ge 2})$ is unramified by \cite[Theorem~9.8]{XuZhao2023Aug}

Next, \cite[Theorem~1.2 and Theorem~1.4]{XuZhao2023Aug} quickly imply (3) and (4),  respectively.
And finally, the three motivic MMVs in (5) are unramified by Lemma~\ref{lem:3SpecialMMVs}.

To show that all the motivic triple mixed values of with irregular parity patterns are ramified except for the three listed in
(5), we will have to analyze case by case in the next section. This will complete the proof of the theorem.
\end{proof}

\section{Ramified motivic triple mixed values}\label{sec:RamDepth3}
We prove in this section that a motivic triple mixed value with irregular parity pattern is ramified if it is not in the list of
Theorem~\ref{thm:triUnramMMV}(5) by consider the following four irregular parity patterns
\begin{center}
(I) $\bfgs=(1,-1,-1)$, (II) $\bfgs=(1,1,-1)$, (III) $\bfgs=(-1,1,1)$, and (IV) $\bfgs=(-1,-1,1)$
\end{center}
in Theorems~\ref{thm:UnramTripleCaseI}--\ref{thm:UnramTripleCaseIV}, respectively.

\begin{thm}\label{thm:UnramTripleCaseI}
All triple mixed values $M^\fm(\bfs;1,-1,-1)$ are ramified.
\end{thm}
\begin{proof}
Let $a'=a+1, b'=b+1, c'=c+1\in\N$ and let $w=a+b+c+3$ be the weight. Then
\begin{equation*}
M^\fm(a',b',c';1,-1,-1)=M^\fm(a',\check{b'},\check{c'})=\sum_{\eta_j=\pm 1} \eta_2
I^\fm(0;\eta_1,0_a,\eta_2,0_b,\eta_3,0_c;1).
\end{equation*}
Our proof strategy is to choose ramified coactions carefully depending on the parities of $a,b,c$ with the following
four different cases:
\begin{center}
(I) $2|w$, $c>0$;\quad
(II) $2|w$, $c=0$;\quad
(III) $2\nmid w$, $2|b$;\quad
(IV) $2\nmid w$, $2\nmid b$.
\end{center}

\medskip
\noindent
\underline{Case (I)}: $n:=w-1$ is odd and $c>0$. Then
\begin{align*}
D_n M^\fm(a',\check{b'},\check{c'})=&\, \sum_{\eta_j=\pm 1} \eta_2 I^\fl(\eta_1;0_a,\eta_2,0_b,\eta_3,0_c;1)\ot
I^\fm(0;\eta_1;1)\\
=&\, \sum_{\eta_2,\eta_3=\pm 1} \eta_2 \Big[I^\fl(-1;0_a,\eta_2,0_b,\eta_3,0_c;0)
+I^\fl(0;0_a,\eta_2,0_b,\eta_3,0_c;1)\Big]\ot I^\fm(0;-1;1) \\
=&\, \sum_{\eta_2,\eta_3=\pm 1} \eta_2 \Big[I^\fl(0;0_c,\eta_3,0_b,\eta_2,0_a;1)
+I^\fl(0;0_a,\eta_2,0_b,\eta_3,0_c;1)\Big]\ot I^\fm(0;-1;1) \\
=&\,-\Big[S^\fl_c(b',a')+t^\fl_a(b',c')\Big]\ot \log^\fm(2)\\
 =&\, \bigg[-(-1)^{b+c} \binom{n-1}{a}+(-1)^a\binom{n-1}{a} \bigg]T^\fl(n) \ot \log^\fm(2)\\
 =&\, 2(-1)^a \binom{n-1}{a} T^\fl(n) \ot \log^\fm(2) \ne 0
\end{align*}
by Lemma~\ref{lem:S_a(b+1,c+1)}. This implies that $M^\fm(a',\check{b'},\check{c'})$ must be ramified.

\medskip
\noindent
\underline{Case (II)}: $n=w-1$ is odd and $c=0$. Then
\begin{align*}
D_n M^\fm(a',\check{b'},\check{c'})=&\,\sum_{\eta_j=\pm 1} \eta_2 D_n I^\fm(0;\eta_1,0_a,\eta_2,0_b,\eta_3,0_c;1)\\
=&\, \sum_{\eta_j=\pm 1} \eta_2 \Big[
I^\fl(0;\eta_1,0_a,\eta_2,0_b;\eta_3)\ot I^\fm(0;\eta_3;1)
+ I^\fl(\eta_1;0_a,\eta_2,0_b,\eta_3;1)\ot I^\fm(0;\eta_1;1) \Big].
\end{align*}
Hence, besides the term as in Case (I) the additional term is
\begin{align*}
&\,\sum_{\eta_j=\pm 1} \eta_2  I^\fl(0;\eta_1,0_a,\eta_2,0_b;\eta_3)\ot I^\fm(0;\eta_3;1) \\
= &\, \sum_{\eta_j=\pm 1} \eta_2  I^\fl(0;\eta_1,0_a,\eta_2,0_b;1)\ot \eta_3 I^\fm(0;\eta_3;1) \\
= &\,S^\fl(a',b')\ot \log^\fm(2)=(-1)^a \binom{n-1}{b} T^\fl(n) \ot \log^\fm(2).
\end{align*}
Therefore,
\begin{align*}
D_n M^\fm(a',\check{b'},\check{c'})=&\,(-1)^{a}\bigg[2\binom{n-1}{a}+\binom{n-1}{b} \bigg] T^\fl(n) \ot \log^\fm(2)\ne 0.
\end{align*}
This implies that $M^\fm(a',\check{b'},\check{c'})$ is ramified again.

\medskip
\noindent
\underline{Case (III)}: $w$ is odd and $b$ is even. Then $a+c$ is even and
\begin{alignat*}{2}
&\,&& D_{b+1} M^\fm(a',\check{b'},\check{c'}) = \sum_{\eta_j=\pm 1} \eta_2 D_{b+1}
I^\fm(0;\eta_1,0_a,\eta_2,0_b,\eta_3,0_c;1)\\
=&\,\gd_{a=b} &&\sum_{\eta_j=\pm 1} \eta_2  I^\fl(0;\eta_1,0_a;\eta_2)\ot I^\fm(0;\eta_2,0_b,\eta_3,0_c;1) \\
+&\,\gd_{0<a\le b} &&\sum_{\eta_j=\pm 1} \eta_2  I^\fl(\eta_1;0_a,\eta_2,0_{b-a};0)\ot I^\fm(0;\eta_1,0_a,\eta_3,0_c;1) \\
+&\,\gd_{a=0} &&\sum_{\eta_j=\pm 1} \eta_2  I^\fl(\eta_1;\eta_2,0_b;\eta_3)\ot I^\fm(0;\eta_1,\eta_3,0_c;1) \\
+&\,\gd_{a>0} &&\sum_{\eta_j=\pm 1} \eta_2  I^\fl(0;\eta_2;0_b;\eta_3)\ot I^\fm(0;\eta_1,0_a,\eta_3,0_c;1) \\
+&\,\gd_{c>0} &&\sum_{\eta_j=\pm 1} \eta_2  I^\fl(\eta_2;0_b,\eta_3;0)\ot I^\fm(0;\eta_1,0_a,\eta_2,0_c;1)  \\
+&\,\gd_{c=0} &&\sum_{\eta_j=\pm 1} \eta_2  I^\fl(\eta_2;0_b,\eta_3;1)\ot I^\fm(0;\eta_1,0_a,\eta_2;1)  \\
+&\,\gd_{b\ge c>0} &&\sum_{\eta_j=\pm 1} \eta_2  I^\fl(0;0_{b-c},\eta_3,0_c;1)\ot I^\fm(0;\eta_1,0_a,\eta_2,0_c;1) \\
=&\,\gd_{a=b} && \tz^\fl(b+1)\ot t^\fm(b+1,c+1) \\
-&\,\gd_{0<a\le b} && T_{b-a}^\fl(a+1)\ot t^\fm(a+1,c+1) \\
+&\,\gd_{a=0} &&\Big[ T^\fl(b+1)\ot S^\fm(1,c+1)-T_{b}^\fl(1)\ot t^\fm(1,c+1) \Big] \\
+&\,\gd_{a>0} && T^\fl(b+1)\ot S^\fm(a+1,c+1)\\
-&\,\gd_{c>0} && \tz^\fl(b+1)\ot S^\fm(a+1,c+1) \\
+&\,\gd_{c=0} && \Big[\tz_b^\fl(1)\ot S^\fm(a+1,1) -\tz^\fl(b+1)\ot S^\fm(a+1,1) \Big] \\
+&\,\gd_{b\ge c>0} && \tz_{b-c}^\fl(c+1)\ot S^\fm(a+1,c+1)  \\
=&\,&& \tz^\fl(b+1)\ot X
\end{alignat*}
where
\begin{align*}
X=&\,\bigg(\gd_{a=b} - (-1)^a \binom{b+a}{b-a} (2^{b+1}-1)\bigg) t^\fm(a+1,c+1) \\
&\, \qquad\qquad + \Big(2^{b+1}-2+\gd_{b\ge c\ge 0}\binom{b}{c} \Big) S^\fm(a+1,c+1).
\end{align*}
We now show that $X$ is ramified by considering three different cases:
\begin{center}
    (III.1) $a,c>0$; \quad (III.2) $a=0$;\quad   (III.3) $a>0,$ $c=0.$
\end{center}

\medskip \noindent
(III.1) If $a,c>0$ then $t^\fm(a+1,c+1)$ is unramified by Lemma~\ref{lem:D1S1bfl}\ref{lem:MtVUnram}
and $S^\fm(a+1,c+1)$ is ramified by Theorem~\ref{thm:triUnramMMV}(4) (and it is  non-vanishing because of coefficient is clearly
positive).

\medskip \noindent
(III.2) If $a=0$ then $\gd_{a=b}=0$ since $b$ is odd while $c$ is even. Thus,
\begin{equation*}
X=-(2^{b+1}-1) t^\fm(1,c+1)+  \Big(2^{b+1}-2+\gd_{b\ge c\ge 0}\binom{b}{c}\Big) S^\fm(1,c+1)
\end{equation*}
is ramified since
\begin{align*}
D_1 X=\log^\fl 2\ot \Big[ -(2^{b+1}-1)(2-\gd_{c=0})   -2\Big(2^{b+1}-2+\gd_{b\ge c\ge 0}\binom{b}{c}\Big) \Big] T^\fm(c+1)
\end{align*}
by Lemma~\ref{lem:D1S1bfl}\ref{lemCase:D1MSV} is nonzero. Indeed, let $f(b,c)$ be the coefficient in front of $T^\fm(c+1)$.
Then
\begin{align*}
f(b,c)=\left\{
  \begin{array}{ll}
    -3(2^{b+1}-1)    \qquad & \hbox{if $c=0$;} \\
    6-2^{b+3}-2\gd_{b\ge c}\binom{b}{c} \qquad        & \hbox{if $c>0$.}
  \end{array}
\right.
\end{align*}
Hence, $f(b,c)<0$ always holds.

\medskip \noindent
(III.3) If $a>0$ and $c=0$ then again $\gd_{a=b}=\gd_{b\ge c>0}=0$ and therefore
\begin{equation*}
   X=- \binom{b+a}{b-a} (2^{b+1}-1) t^\fm(a+1,1)+  (2^{b+1}-1) S^\fm(a+1,1)
\end{equation*}
is ramified since by Lemma~\ref{lem:D1S1bfl}\ref{lemCase:D1MSV(k,1)}
\begin{equation*}
D_1 X= (2^{b+1}-1) \bigg[\binom{b+a}{b-a}(2^{a+1}-1)-1\bigg]  \log^\fl 2\ot  \tz^\fm(a+1)\ne 0.
\end{equation*}
From (III.1)--(III.3) we see that $X$ and therefore $D_{b+1} M^\fm(a',\check{b'},\check{c'})$ is always ramified.
Hence, $M^\fm(a',\check{b'},\check{c'})$ is ramified in Case (III).

\medskip
\noindent
\underline{Case (IV)}: both $w$ and $b$ are odd. Then $a+c$ is also odd.
We first handle a special case: $a=0$. Then we have $b,c>0$ since they are both odd. Thus,
\begin{align*}
&\, D_1 M^\fm(a',\check{b'},\check{c'}) = \sum_{\eta_j=\pm 1} \eta_2 D_1 I^\fm(0;\eta_1,\eta_2,0_b,\eta_3,0_c;1)\\
=&\, \sum_{\eta_j=\pm 1} \eta_2 \Big[  I^\fl(0;\eta_1;\eta_2)\ot I^\fm(0;\eta_2,0_b,\eta_3,0_c;1)+  I^\fl(\eta_1;\eta_2;0)\ot
I^\fm(0;\eta_1,0_b,\eta_3,0_c;1) \Big]  \\
=&\, \sum_{\eta_j=\pm 1}  \Big[ I^\fl(0;\eta_1;1)\ot \eta_2 I^\fm(0;\eta_2,0_b,\eta_3,0_c;1) -  \eta_2  I^\fl(0;\eta_2;1)\ot
\eta_1 I^\fm(0;\eta_1,0_b,\eta_3,0_c;1)\Big]  \\
=&\, -2 \log^\fl 2 \ot t^\fm(b+1,c+1)
\end{align*}
which is ramified.

Now we assume $a>0$. Then,
\begin{alignat*}{2}
D_{b+2} M^\fm&\,(a',\check{b'},\check{c'}) && = \sum_{\eta_j=\pm 1} \eta_2 D_{b+2}
I^\fm(0;\eta_1,0_a,\eta_2,0_b,\eta_3,0_c;1)\\
=&\,\gd_{a=b+1} &&\sum_{\eta_j=\pm 1} \eta_2  I^\fl(0;\eta_1,0_a;\eta_2)\ot I^\fm(0;\eta_2,0_b,\eta_3,0_c;1) \\
+&\,\gd_{1<a\le b+1} &&\sum_{\eta_j=\pm 1} \eta_2  I^\fl(\eta_1;0_a,\eta_2,0_{b+1-a};0)\ot
I^\fm(0;\eta_1,0_{a-1},\eta_3,0_c;1) \\
+&\,\gd_{a=1} &&\sum_{\eta_j=\pm 1} \eta_2  I^\fl(\eta_1;0,\eta_2,0_b;\eta_3)\ot I^\fm(0;\eta_1,\eta_3,0_c;1) \\
+&\,\gd_{a>1} &&\sum_{\eta_j=\pm 1} \eta_2  I^\fl(0;0,\eta_2;0_{b};\eta_3)\ot I^\fm(0;\eta_1,0_{a-1},\eta_3,0_c;1) \\
+&\,\gd_{c=1} &&\sum_{\eta_j=\pm 1} \eta_2  I^\fl(\eta_2;0_b,\eta_3,0;1)\ot I^\fm(0;\eta_1,0_a,\eta_2;1)  \\
+&\,\gd_{c>1} &&\sum_{\eta_j=\pm 1} \eta_2  I^\fl(\eta_2;0_b,\eta_3,0;0)\ot I^\fm(0;\eta_1,0_a,\eta_2,0_{c-1};1)  \\
+&\,\gd_{b+1\ge c>1} &&\sum_{\eta_j=\pm 1} \eta_2  I^\fl(0;0_{b+1-c}, \eta_3,0_c;1)\ot I^\fm(0;\eta_1,0_a,\eta_2,0_{c-1};1) \\
=&\,\gd_{a=b+1} && \tz^\fl(b+2)\ot t^\fm(b+1,c+1) \\
-&\,\gd_{1\le a\le b+1} && T_{b+1-a}^\fl(a+1)\ot t^\fm(a,c+1) \\
+&\,\gd_{a\ge 1} && T_1^\fl(b+1)\ot S^\fm(a,c+1)\\
+&\,\gd_{c=1} &&\tz_b^\fl(2)\ot S^\fm(a+1,c) \\
-&\,\gd_{c\ge 1} && \tz_1^\fl(b+1)\ot S^\fm(a+1,c) \\
+&\,\gd_{b+1\ge c>1} && \tz_{b+1-c}^\fl(c+1)\ot S^\fm(a+1,c)  \\
=&\,&& \tz^\fl(b+2)\ot X
\end{alignat*}
where
\begin{align*}
X=&\,
 \Big(\gd_{b+1\ge c>1} \binom{b+1}{c}+(\gd_{c=1}-1)(b+1)\Big) S^\fm(a+1,c)\\
&  +(b+1)(2^{b+2}-1) S^\fm(a,c+1)+ \bigg(\gd_{a=b+1} + (-1)^a \binom{b+1}{a} (2^{b+2}-1)\bigg) t^\fm(a,c+1) .
\end{align*}
Notice that if $a\ge 2,c\ge 1$ then $t^\fm(a,c+1)$ is unramified by Lemma~\ref{lem:D1S1bfl}\ref{lem:MtVUnram}
and both $S^\fm(a,c+1)$ and $S^\fm(a+1,c)$ are ramified by Lemma~\ref{lem:D1S1bfl}\ref{lemCase:DTDS} since $a$ and $c$ have
different parities. More precisely,
\begin{align*}
 D_{a+c} X
=&\, \bigg\{  \Big(\gd_{b+1\ge c>1} \binom{b+1}{c}+(\gd_{c=1}-1)(b+1)\Big) \bigg[\delta_{c=1}-(-1)^a
2\binom{a+c-1}{a}(2^{a+c}-1)\bigg] \\
&\,  +\gd_{a>0} (b+1)(2^{b+2}-1)\bigg[\delta_{c=0}+(-1)^a 2\binom{a+c-1}{a-1}(2^{a+c}-1)\bigg] \\
&\, +\gd_{a>0,c=0}\bigg(\gd_{a=b+1} + (-1)^a \binom{b+1}{a} (2^{b+2}-1)\bigg) (2^{a+c}-1)  \bigg\} \tz^\fl(a+c)\ot \log^\fm 2
.
\end{align*}
We now show the coefficient is nonzero by considering the following three cases:
\begin{center}
   (IV.1) $a>0,c=1$;\quad  (IV.2) $a>0$, $c=0$;    \quad   (IV.3) $a>0$, $c\ge 2.$
\end{center}

\medskip \noindent
(IV.1) If $a>0$ and $c=1$ then only the middle term appears and therefore $X$ is unramified.

\medskip
 \noindent
(IV.2) If $a>0$ and  $c=0$, then the first term does not appear, both $a$ and $b$ are odd so that $\gd_{a=b+1}=0$. Thus,
the coefficient is 0 if and only if
\begin{equation*}
    (b+1)(2^{b+2}-1)\bigg[1-2\binom{a+c-1}{a-1}(2^{a+c}-1)\bigg]
    -\binom{b+1}{a} (2^{b+2}-1)(2^{a+c}-1)=0,
\end{equation*}
which is impossible since the number on the left is clearly negative.

\medskip \noindent
(IV.3) Now we assume $a>0$ and $c\ge 2$. If $b+1\ge c$ then $D_{a+c} X=0$ if and only if
\begin{align*}
    (b+1)(2^{b+2}-1)(-1)^a 2\binom{a+c-1}{a-1}(2^{a+c}-1) =&\,
    \bigg(\binom{b+1}{c}-b-1\bigg) (-1)^a 2\binom{a+c-1}{a}(2^{a+c}-1),\\
    (b+1)(2^{b+2}-1) \binom{a+c-1}{a-1} =&\,  \bigg(\binom{b+1}{c}-b-1\bigg) \binom{a+c-1}{a},\\
    a(2^{b+2}-1) =&\,  \binom{b}{c-1}-1 .
\end{align*}
Noticing that $\binom{b}{c-1}\le 2^b$ we see the above equation has no solution for $a\ge 1.$

If $c\ge 2$ and $b+1< c$ then $D_{a+c} X=0$ if and only if
\begin{equation*}
    (b+1)(2^{b+2}-1)(-1)^a 2\binom{a+c-1}{a-1}(2^{a+c}-1) =
    -(b+1) (-1)^a 2\binom{a+c-1}{a}(2^{a+c}-1)
\end{equation*}
i.e.
\begin{equation*}
    (2^{b+2}-1) \binom{a+c-1}{a-1} = -  \binom{a+c-1}{a},
\end{equation*}
which is impossible again since the two sides have different signs.

In all the cases (IV.1)--(IV.3), we have seen that $D_{a+c} X$ is always ramified which implies that $D_{b+2}
M^\fm(a',\check{b'},\check{c'})$ is ramified and
so is $M^\fm(a',\check{b'},\check{c'})$ by Theorem~\ref{thm:Glanois}.

Combining Cases (I)--(IV) we have finished the proof of the theorem.
\end{proof}

\begin{thm}\label{thm:UnramTripleCaseII}
All triple mixed values $M^\fm(\bfs;1,1,-1)$ are ramified.
\end{thm}

\begin{proof}
The proof uses the same ideas as in the proceeding proof and therefore will be only sketched below.
Let $a'=a+1, b'=b+1, c'=c+1\in\N$ and let $w=a+b+c+3$. Then
\begin{equation*}
M^\fm(a',b',c';1,1,-1)=M^\fm(a',b',\check{c'})=\sum_{\eta_j=\pm 1} \eta_3 I^\fm(0;\eta_1,0_a,\eta_2,0_b,\eta_3,0_c;1).
\end{equation*}
We consider the same four cases as in the previous proof:
\begin{center}
(I) $2|w$, $c>0$;\quad
(II) $2|w$, $c=0$;\quad
(III) $2\nmid w$, $2|b$;\quad
(IV) $2\nmid w$, $2\nmid b$.
\end{center}

\medskip
\noindent
\underline{Case (I)}: $n=w-1$ is odd and $c>0$. Then
\begin{align*}
D_n M^\fm(a',b',\check{c'})=&\, 2(-1)^c \binom{n-1}{c} T^\fl(n) \ot \log^\fm(2) \ne 0
\end{align*}
which is exactly the opposite to Case (I) with $a$ and $c$ switched in Theorem~\ref{thm:UnramTripleCaseI}. This implies that
$M^\fm(a',b',\check{c'})$ must be ramified.

\medskip
\noindent
\underline{Case (II)}: $n=w-1$ is odd and $c=0$. Then
\begin{align*}
D_n M^\fm(a',b',\check{c'})=&\,\bigg[2^{n+1}-3+(-1)^a\binom{n}{a+1}\bigg] \tz^\fl(n) \ot \log^\fm(2)\ne 0
\end{align*}
since $2^n\ge \binom{n}{0}+\binom{n}{a+1}+\binom{n}{n}$ for all $a+1<n$. This implies that $M^\fm(a',b',\check{c'})$ is
ramified again.

\medskip
\noindent
\underline{Case (III)}: $w$ is odd and $b$ is even. Then $a+c$ is even. Then
\begin{align*}
D_{b+1} M^\fm(a',b',\check{c'}) =& \tz^\fl(b+1)\ot  \bigg(\gd_{0<a<b} (-1)^{a+1}\binom{b}{a}+\gd_{a>0}+1-2^{b+1}\bigg)
S^\fm(a+1,c+1) \\
&\,  \qquad\qquad + \gd_{b\ge c}(-1)^c (2^{b+1}-1)\binom{b}{c}\tz^\fm(a+1,c+1).
\end{align*}
Then it can be shown that $D_{b+1} M^\fm(a',b',\check{c'})$ is ramified and so is $M^\fm(a',b',\check{c'})$  in this case.

\medskip
\noindent
\underline{Case (IV)}: both $w$ and $b$ are odd. Then $a+c$ is also odd. Thus,
\begin{align*}
D_{b+2} M^\fm(a',b',\check{c'}) =&\,\tz^\fl(b+2)\ot   \bigg(\gd_{a=b+1}-(-1)^a \gd_{1\leq a \leq b+1}\binom{b+1}{a} -\gd_{a\ge
1} (b+1)\bigg)S^\fm(a,c+1)\\
&\,  \qquad\qquad + \gd_{a=0}\bigg(2^{b+2}+b \bigg) t^\fm(c+1)
+\gd_{c\ge 1} (2^{b+2}-1)(b+1) S^\fm (a+1,c)\\
&\, \qquad\qquad +\gd_{b+1\ge c\ge 1} (-1)^c (2^{b+2}-1)\binom{b+1}{c}\tz^\fm(a+1,c).
\end{align*}
By considering the following five cases
\begin{center}
    (i) $a=c=0$, \quad(ii) $a=0,\ c=1$, \quad(iii) $a=0,\ c\ge 2$, \quad(iv) $a>0, \ c=0 $, \quad(v) $a>0, \ c\ge 1$,
\end{center}
we can show that $D_{b+2} M^\fm(a',b',\check{c'})$ is always ramified and so is $M^\fm(a',b',\check{c'})$.
This finishes the proof of the theorem.
\end{proof}

\begin{thm}\label{thm:UnramTripleCaseIII}
The motivic triple mixed value $M^\fm(\bfs;-1, -1, 1)$ is unramified only when $\bfs=(1, 1, 3)$.
\end{thm}

\begin{proof}
Let $a'=a+1, b'=b+1, c'=c+1\in\N$ and let $w=a+b+c+3$. Then
\begin{equation*}
M^\fm(a',b',c';-1,-1,1)=M^\fm(\check{a'},\check{b'},c')=\sum_{\eta_j=\pm 1} \eta_1\eta_3
I^\fm(0;\eta_1,0_a,\eta_2,0_b,\eta_3,0_c;1).
\end{equation*}
We now consider the following cases:
\begin{center}
(I) $2|w$, $c>0$;\quad
(II) $2|w$, $c=0$;\quad
(III) $2\nmid w$, $2\nmid b$;\quad
(IV) $2\nmid w$, $a=b=0$, $2|c\ge 4$; \\
(V) $2\nmid w$, $2|b$, and either (1) $a=b=c=0$, or (2) $0<a=b<c$, or (3) $a=0, b\ge \max\{c,1\}$;
(VI) $2\nmid w$, $2|b$,  but not in Cases (IV) or (V).
\end{center}

\medskip
\noindent
\underline{Case (I)}: $n=w-1$ is odd and $c>0$. Then
\begin{align*}
D_n M^\fm(\check{a'},\check{b'},c')=&\, -2(-1)^c \binom{n-1}{c} T^\fl(n) \ot \log^\fm(2) \ne 0.
\end{align*}

\medskip
\noindent
\underline{Case (II)}: $n=w-1$ is odd and $c=0$. Then
\begin{align*}
D_n M^\fm(\check{a'},\check{b'},c')=&\,-\bigg[2(-1)^{c}\binom{n-1}{c}+1 \bigg] T^\fl(n) \ot \log^\fm(2)\ne 0.
\end{align*}

\medskip
\noindent
\underline{Case (III)}: both $w$ and $b$ are odd.

\medskip
\noindent
(III.1) $a=0$. Then $c$ is odd and
\begin{align*}
 D_1 M^\fm(1,\check{b'},\check{c'})=&\,\log^\fl 2 \ot\Big[  T^\fm(b+1,c+1)+t^\fm(b+1,c+1) \Big].
\end{align*}
Since $t^\fm(b+1,c+1)$ is unramified by Lemma~\ref{lem:D1S1bfl}\ref{lem:MtVUnram} and $T^\fm(b+1,c+1)$ is ramified by
Lemma~\ref{lem:D1S1bfl}\ref{lemCase:DTDS}, we see that $M^\fm(1,\check{b'},\check{c'})$ is ramified.

\medskip
\noindent
(III.2) $a>0$.  Then $a+c$ is odd and
\begin{align*}
D_{b+2} M^\fm(\check{a'},\check{b'},c')
= &\, \tz^\fl(b+2)\ot   \bigg[\gd_{c\ge 1} (b+1)(2^{b+2}-1) T^\fm(a+1,c) \\
&\, +\bigg(\gd_{a=b+1}(2^{b+2}-1) -\gd_{a\le b+1}  (-1)^a \binom{b+1}{a} -b-1 \bigg) T^\fm(a,c+1)\\
&\, +  (2^{b+2}-1)\Big(\gd_{b+1\ge c>0}(-1)^c\binom{b+1}{c}-\gd_{c=1}(b+1) \Big) t^\fm(a+1,c)
 \bigg].
\end{align*}
By considering three cases
\begin{center}
    (III.2.i) $c=0$,\quad (III.2.ii) $c=1$,\quad (III.2.iii) $c\ge 2$
\end{center}
we can show that $D_{b+2} M^\fm(\check{a'},\check{b'},c')$ is always ramified and so is
$M^\fm(\check{a'},\check{b'},c')$ in Case (III).

\bigskip
The following Cases (IV)--(VI) handle the case when $w$ is odd and $b$ is even.

\medskip
\noindent
\underline{Case (IV)}: $w$ is odd, $a=b=0$, and $c\ge 4$ is even. Then
\begin{align*}
 D_3 M^\fm(\check{1},\check{1},c')
=&\,- t_1^\fl(1,1)\ot \tz^\fm(c)-T_2^\fl(1)\ot T^\fm(1,c-1)
\end{align*}
which is ramified by Lemma~\ref{lem:D1S1bfl}\ref{lemCase:DTDS}.

\medskip
\noindent
\underline{Case (V)}: $w$ is odd, $b$ is even, and either (1) $a=b=c=0$, or (2) $0<a=b<c$, or (3) $a=0,b\ge
\max\{c,1\}$. Then
\begin{align*}
D_{c+1} M^\fm(\check{a'},\check{b'},c')
=&\, \tz^\fl(c+1)\ot\bigg[\bigg(\gd_{a+b\ge c\ge b}\binom{c}{b}-\gd_{a< c\le a+b}\binom{c}{a}\bigg)  T^\fm(a+b-c+1,c+1)  \\
+&\,\gd_{c-a>b} \binom{c}{b}(2^{c+1}-1) T^\fm(a+b+2)\\
+&\, \gd_{c\ge b} \binom{c}{b}(2^{c+1}-1) T^\fm(a+1,b+1)
+ (2^{c+1}-1)  t^\fm(a+1,b+1) \bigg].
\end{align*}
By consider the coaction $D_{a+b+1}$ on the right factor we can show that it is always ramified.

\medskip
\noindent
\underline{Case (VI)}: $w$ is odd, $b$ is even, but not in Cases (IV) or (V). Then
\begin{alignat*}{2}
&\,D_{b+1} M^\fm(a',b',\check{c'})
=&\, C\, \tz^\fl(b+1)\ot T^\fm (a+1,c+1)-\gd_{b\ge c\ge 0} \binom{b}{c} T^\fl(b+1)\ot t^\fm(a+1,c+1)
\end{alignat*}
where $T^\fl(b+1)\ot t^\fm(a+1,c+1) \in\calL_{b+1}\ot \calH_{a+c+2}$ and the number
\begin{align*}
C=&\, \gd_{a=b} (2^{b+1}-1)- (-1)^a\binom{b}{a}
+ 2 - 2^{b+1}+(-1)^a\gd_{b\ge c}\binom{b}{a} (2^{b+1}-1).
\end{align*}
By considering the following cases not covered by Cases (IV) or (V):
\begin{center}
    (1) $2\nmid a$;\quad (2) $2|a$, $0<a\ne b\ge c$;  \quad
    (3) $2|a$, $0<a\ne b< c$;\quad (4) $a=b\ge c$, (5) $a=0<b<c$,
\end{center}
we can show that $C\ne 0$ and therefore by Lemma~\ref{lem:D1S1bfl}\ref{lemCase:DTDS}
$D_{b+1} M^\fm(a',b',\check{c'})$ is always ramified in Case (VI).

\bigskip
Combining Cases (I)--(VI) we have finished the proof of the theorem.
\end{proof}

\begin{thm}\label{thm:UnramTripleCaseIV}
The motivic triple mixed values $M^\fm(\bfs;-1, 1, 1)$ are unramified only when $\bfs=(1, 2, 2)$ or $\bfs=(1, 4, 4)$.
\end{thm}

\begin{proof}
Let $a'=a+1, b'=b+1, c'=c+1\in\N$. Then
\begin{equation*}
M^\fm(a',b',c';-1,1,1)=M^\fm(\check{a'},b',c')=\sum_{\eta_j=\pm 1} \eta_1\eta_2 I^\fm(0;\eta_1,0_a,\eta_2,0_b,\eta_3,0_c;1).
\end{equation*}

\medskip
\noindent
\underline{Case (I)}: $n=w-1$ is odd and $c>0$. Then
\begin{align*}
D_n M^\fm(\check{a'},b',c')=&\,-2(-1)^a \binom{n-1}{a} T^\fl(n) \ot \log^\fm(2) \ne 0.
\end{align*}

\medskip
\noindent
\underline{Case (II)}: $n=w-1$ is odd and $c=0$. Then
\begin{align*}
D_n M^\fm(\check{a'},b',c')=&\,-3(-1)^a \binom{n-1}{a}T^\fl(n)\ot \log^\fm(2)\ne 0.
\end{align*}

\medskip
\noindent
\underline{Case (III)}: both $w$ and $b$ are odd. We first handle a special case.

\medskip
\noindent
(III.1). If $a=0$ and $b=c\ge 5$ then
\begin{align*}
D_{b+4} M^\fm (\check{1},b',b')&=T^{\fl}(b+4)\ot \tz^\fm (b-1)+\binom{b+3}{3}\tz(b+4)\ot T^\fm (1,b-2)
\end{align*}
which is clearly ramified.

\medskip
\noindent
(III.2). Opposite cases of (III.1) under the restriction that $a+b+c$ is even and $b$ is odd.  Namely,
either $a\ne 0$ or $b\ne c$.  Then
\begin{align*}
&\,D_{b+2} M^\fm(\check{a'},b',c') = \tz^\fl(b+2)\ot  (X'+X)
\end{align*}
where
\begin{align*}
X'=&\,  \gd_{a=b+1}(2^{a+1}-1)\tz^\fm(b+1,c+1)+\gd_{a=0} 2(2^{b+2}-1)\tz^\fm(c+1)\\
&\qquad -\gd_{1\leq a \leq b+1} (-1)^a (2^{b+2}-1)\binom{b+1}{a}\tz^\fm(a,c+1),\\
X=&\, \bigg(\gd_{b+1\ge c\ge 1}(-1)^c \binom{b+1}{c}+\gd_{c\ge 1}(b+1) \bigg)T^\fm (a+1,c)\\
& \qquad-\gd_{a\ge 1} (2^{b+2}-1)(b+1)T^\fm (a,c+1)  .
\end{align*}
Clearly, $X'$ is unramified. Since $a+c$ is odd, by Lemma~\ref{lem:D1S1bfl}\ref{lemCase:DTDS} it remains to show that the two
ramified terms in $X$ do not cancel each other. But we can show that
$D_{a+c}X$ by the same idea as used before. This completes the proof of Case (III).

\medskip
\noindent
\underline{Case (IV))}: both $w$ and $b$ are even. Then $a+c$ is even. Then
\begin{align*}
&\,D_{b+1} M^\fm(\check{a'},b',c') =&\, &&\Big[\gd_{a=b}  t^\fl(b+1)
- \gd_{a\le b} t_{b-a}^\fl(a+1)\Big]\ot \tz^\fm(a+1,c+1) + \tz^\fl(b+1) \ot X,
\end{align*}
where
\begin{align*}
X=&\,
\bigg(2^{b+1}-2+\gd_{b\ge c} (-1)^c\binom{b}{c} \bigg)T^\fm(a+1,c+1).
\end{align*}
Since $\binom{b}{c} \le 2^b$, it is evident that the coefficient in front of $T^{\fm}(a+1, c+1)$ is nonzero. Furthermore, as
$a+c$ is even, $T^{\fm}(a+1, c+1)$ is ramified, which in turn implies that $M^{\fm}(\check{a'},b',c')$ must be ramified.

Combining Cases (I)--(IV) we have finished the proof of the theorem.
\end{proof}

\section{Two families of unramified MMVs of arbitrary depth}\label{sec:2family}
Beyond depth three, we have searched extensively for other unramified MMVs and found three infinite families of such values.
In this section, we present two such families with the parity signature patterns $(1_{a-1},-1,1_{n-a})$ for all $1\le a<n$ in
Theorem~\ref{thm:UnramMMVfamily1} and $(\{-1\}_{a-1},1,\{-1\}_{n-a})$ for all $1<a<n$ in Theorem~\ref{thm:UnramMMVfamily2}. Here,
$1_n$ means that the number $1$ is repeated $n$ times. The same notation scheme will be applied to other singletons such as
$2$  and  $\check{2}$ in the next lemma.

\begin{lem}\label{lem:LieUnram}
Let $n\ge 1$. Then the following two families of motivic MMVs are both unramified, namely, they are both in $\calL_{2n+1}$:
\begin{equation*}
\emph{(1)}\  M^\fl(2_n, \check{1}),\qquad
\emph{(2)}\  M^\fl(\check{2}_n,1).
\end{equation*}
\end{lem}

\begin{proof}
By definition,
\begin{align*}
 M^\fl(2_n, \check{1})=&\, \sum_{\eta_j=\pm 1} \eta_{n+1} I^\fl(0;\eta_1,0,\dots, \eta_n,0,\eta_{n+1};1),\\
 M^\fl(\check{2}_n,1)=&\, \sum_{\eta_j=\pm 1} \eta_1\eta_{n+1} I^\fl(0;\eta_1,0,\dots, \eta_n,0,\eta_{n+1};1).
\end{align*}
When $n=1$, we have from Lemma~\ref{lem:DBMMVs}
\begin{align*}
 M^\fl(2, \check{1})=&\, S^\fl(2,1)=- T^\fl(3),\qquad
 M^\fl(\check{2},1)=  T^\fl(2, 1)=-2 T^\fl(3)\in\calL_{3}.
\end{align*}
By Lemma~\ref{lem:D1S1bfl}\ref{Murakami-lemma9}, the only nontrivial cutting for $ M^\fm(2_n, \check{1})$ by coaction
$D_{2n-2j+1}$ ($1\le j\le n$) is
\begin{align*}
D_{2n-2j+1}  M^\fm(2_n, \check{1})=&\, \sum_{\eta_j=\pm 1}  \eta_{n+1}D_{2n-2j+1} I^\fm(0;\eta_1,0,\dots,
\eta_n,0,\eta_{n+1};1)\\
=&\, \sum_{\eta_j=\pm 1}  \eta_{n+1} I^\fl(0;\eta_{j+1},0,\dots, \eta_n,0,\eta_{n+1};1)\ot I^\fm(0;\eta_1,0,\dots, \eta_j,0;1)
\\
=&\,  M^\fl(2_{n-j}, \check{1}) \ot \tz^\fm(2_j),
\end{align*}
which is stable by induction assumption, except for $j=n$. But
\begin{equation*}
M^\fl(2_n, \check{1}) =M^\fl(2_n, \check{1})-M^\fl(\check{1}) \tz^\fl(2_n)
\end{equation*}
and since $D_1\tz^\fm(2_n))=0$, by derivation property of $D_1$ we have
\begin{equation*}
    D_1 \Big(M^\fm(2_n, \check{1})-M^\fm(\check{1}) \tz^\fm(2_n))\Big)=0.
\end{equation*}
This shows that $M^\fl(2_n, \check{1})\in\calL_{2n+1}$ is unramified for all $n\ge 1$.

Similarly, by Lemma~\ref{lem:D1S1bfl}\ref{Murakami-lemma9}, the only nontrivial cutting for $ M^\fm(\check{2}_n,1)$ is
\begin{align*}
D_{2n-2j+1} M^\fm(\check{2}_n,1)=&\, \sum_{\eta_j=\pm 1} \eta_1\eta_{n+1}D_{2n-2j+1} I^\fm(0;\eta_1,0,\dots,
\eta_n,0,\eta_{n+1};1)\\
=&\, \sum_{\eta_j=\pm 1} \eta_1\eta_{n+1} I^\fl(0;\eta_{j+1},0,\dots, \eta_n,0,\eta_{n+1};1)\ot I^\fm(0;\eta_1,0,\dots,
\eta_j,0;1) \\
=&\,  M^\fl(2_{n-j}, \check{1}) \ot t^\fm(2_j)
\end{align*}
for $1\le j\le n$, which is stable by induction assumption and the fact that $t^\fm(2_j)\in\calH_{2j}$ by
Lemma~\ref{lem:D1S1bfl}\ref{lem:MtVUnram}, except for $j=n$. But
\begin{equation*}
M^\fl(\check{2}_n,1) =M^\fl(\check{2}_n,1)-M^\fl(\check{1})t^\fl(2_n)
\end{equation*}
and since $D_1t^\fm(2_n)=0$ we get
\begin{equation*}
    D_1 \Big( M^\fm(\check{2}_n,1)-M^\fm(\check{1})t^\fm(2_n)\Big)=0.
\end{equation*}
This shows that $M^\fl(\check{2}_n,1)\in\calL_{2n+1}$ is unramified for all $n\ge 1$. We have completed the proof of the
lemma.
\end{proof}

\begin{thm}\label{thm:UnramMMVfamily1}
Suppose $n\ge 3$. Then the motivic MMVs $M^\fm(2_{a-1},\check{1},2_{n-a})$ are unramified for all $1\le a<n$.
\end{thm}

\begin{proof}
We will prove the theorem by applying induction on $n$. By definition,
\begin{align*}
M^\fm(2_{a-1},\check{1},2_{n-a})=&\, \sum_{\eta_j=\pm 1} \eta_a\eta_{a+1} I^\fm(0;\eta_1,0,\dots, \eta_{a-1},0,
\eta_a,\eta_{a+1},0,\dots,\eta_n,0 ;1).
\end{align*}

When $n=3$, we have two values to consider. First,
\begin{align*}
M^\fm(2,\check{1},2)=&\, S^\fm(2,1,2)\in \calH_5
\end{align*}
by \cite[Theorem~1.4(3)(ii)]{XuZhao2023Aug}. Second, $M^\fm(\check{1},2,2)\in \calH_5$
by Theorem~\ref{thm:triUnramMMV}(5)(ii).

Now we assume $n\ge 4$ and let $r$ be an odd integer. The following picture shows all the positive
ways to cut $M^\fm(2_{a-1},\check{1},2_{n-a})$ when computing its image under the coaction $D_r$.
For convenience, let $w=2n-1$ be its weight.
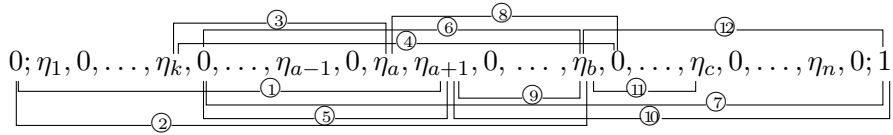
\begin{center}
\begin{tikzpicture}[scale=0.9]
\node (A0) at (0.05,0)
{$0;\eta_1,0,\dots,\eta_k,0,\dots,\eta_{a-1},0,\eta_a,\eta_{a+1},0,\,\dots\,,\eta_b,0,\dots,\eta_c,0,\dots,\eta_n,0;1$};
\node (A1) at (-2.6,-0.4) {${}$};
\node (A2) at (-4.2,-0.9) {${}$};
\node (A3) at (-2.5,0.6) {${}$};
\node (A4) at (-0.6,0.3) {${}$};
\node (A5) at (-1.8,-0.8) {${}$};
\node (A6) at (0,0.5) {${}$};
\node (A7) at (3.95,-0.6) {${}$};
\node (A8) at (0.8,0.7) {${}$};
\node (A9) at (1.25,-0.5) {${}$};
\node (A10) at (3,-0.8) {${}$};
\node (A11) at (2.8,-0.4) {${}$};
\node (A12) at (4.1,0.5) {${}$};
\draw (-6.3,-0.25) to (-6.3,-0.4) to (A1) node {$\cic{1}$} to (-.1,-0.4) to (-.1,-0.25);
\draw (-6.33,-0.25) to (-6.33,-0.9) to (A2) node {$\cic{2}$} to (2.05,-0.9) to (2.05,-0.25);
\draw (-4.02,0.15) to (-4.02,0.6) to (A3) node {$\cic{3}$} to (-.9,0.6) to (-.9,0.15);
\draw (-3.95,0.15) to (-3.95,0.3) to (A4) node {$\cic{4}$} to (2.45,0.3) to (2.45,0.2);
\draw (-3.58,-0.25) to (-3.58,-0.8) to (A5) node {$\cic{5}$} to (0,-0.8) to (0,-0.25);
\draw (-3.58,0.25) to (-3.58,0.5) to (A6) node {$\cic{6}$} to (1.95,0.5) to (1.95,0.15);
\draw (-3.54,-0.25) to (-3.54,-0.6) to (A7) node {$\cic{7}$} to (6.4,-0.6) to (6.4,-0.25);
\draw (-.81,0.15) to (-.81,0.7) to (A8) node {$\cic{8}$} to (2.5,0.7) to (2.5,0.2);
\draw (.17,-0.25) to (.17,-0.5) to (A9) node {$\cic{9}$} to (1.95,-0.5) to (1.95,-0.25);
\draw (.1,-0.25) to (.1,-0.8) to (A10) node {$\cic{\hskip-1.2pt1\!0}$} to (6.5,-0.8) to (6.5,-0.25);
\draw (2.15,-0.25) to (2.15,-0.4) to (A11) node {$\cic{\hskip-1.2pt1\!1}$} to (3.65,-0.4) to (3.65,-0.25);
\draw (2,0.15) to (2,0.5) to (A12) node {$\cic{\hskip-1.2pt1\!2}$} to (6.4,0.5) to (6.4,0.25);
\node (A) at (0,-1.6) {Possible cuts of $D_r M^\fm(2_{a-1},\check{1},2_{n-a})$ ($1\le a<n$) for odd $r$. Here, $1\le k<a<b\le
n$.};
\end{tikzpicture}
\end{center}
Then, we have
\begin{align*}
\ncic{1}=&\, \gd_{r=2a-1}\, \sum \eta_a\eta_{a+1} I^\fl\big(0;\{\eta_j,0\}_{j=1}^{a-1},\eta_a;\eta_{a+1}\big)\ot
I^\fm\big(0;\{\eta_j,0\}_{j=a+1}^n;1\big) \\
=&\, \gd_{r=2a-1}\, \sum \eta_a I^\fl\big(0;\{\eta_j,0\}_{j=1}^{a-1},\eta_a;1\big)\ot
I^\fm\big(0;\{\eta_j,0\}_{j=a+1}^n;1\big)\\
=&\, \gd_{r=2a-1}\, M^\fl(2_{a-1}, \check{1}) \ot \tz^\fm(2_{n-a})
\left\{
  \begin{array}{ll}
     \in \calL_r\ot \calH_{w-r}  \text{ by Lemma~\ref{lem:LieUnram}} \qquad & \hbox{if $a>1$;} \\
     \text{it is canceled by $\ncic{8}$)} & \hbox{if $a=1$,}
  \end{array}
\right.\\
\ncic{2}=&\, \gd_{r=2b-3,b>a+1} \, \sum \eta_a\eta_{a+1}
I^\fl\big(0;\{\eta_j,0\}_{j=1}^{a-1},\eta_a,\{\eta_j,0\}_{j=a+1}^{b-1};\eta_b\big)\ot I^\fm\big(0;\{\eta_j,0\}_{j=b}^n;1\big)
\\
=&\, \gd_{r=2b-3,b>a+1} \, \sum \eta_a\eta_{a+1}
I^\fl\big(0;\{\eta_j,0\}_{j=1}^{a-1},\eta_a,\{\eta_j,0\}_{j=a+1}^{b-1};1\big)\ot I^\fm\big(0;\{\eta_j,0\}_{j=b}^n;1\big) \\
&\, \hskip1cm (\text{by change of variables $\eta_j\to \eta_j\eta_b$ for the first factor})\\
=&\, \gd_{r=2b-3,b>a+1} \,  M^\fl(2_{a-1}, \check{1},2_{b-a-1}) \ot \tz^\fm(2_{n-b+1})\in\calL_r\ot \calH_{w-r} \quad\text{by
induction}, \\
\ncic{3}=&\, \gd_{r=2a-2k-1,k<a}\,  \sum \eta_a\eta_{a+1} I^\fl(\eta_k;0,\{\eta_j,0\}_{j=k+1}^{a-1};\eta_a) \ot
I^\fm\big(0;\{\eta_j,0\}_{j=1}^{k-1},\eta_k,\eta_a,\{\eta_j,0\}_{j=a+1}^n;1\big), \\
=&\, 0 \hskip1cm \text{(by Lemma~\ref{lem:D1S1bfl}\ref{Murakami-lemma9})}, \\
\ncic{4}= &\,\gd_{r=2b-2k-1,b>a>k}\,  \sum \eta_a\eta_{a+1}
I^\fl(\eta_k;\{0,\eta_j\}_{j=k+1}^a,\{\eta_j,0\}_{j=a+1}^{b-1},\eta_b;0) \ot
I^\fm\big(0;\{\eta_j,0\}_{j=1}^{k},\{\eta_j,0\}_{j=b+1}^n;1\big) \\
= &\,-\gd_{r=2b-2k-1,b>a>k}\, M^\fl(2_{b-a-1}, \check{1},2_{a-k}) \ot \tz^\fm(2_{n-b+k})\in\calL_r\ot \calH_{w-r}
\quad\text{by induction}, \\
\ncic{5}=&\, \gd_{r=2a-2k-1,a>k}\,  \sum \eta_a\eta_{a+1} I^\fl(0;\{\eta_j,0\}_{j=k+1}^{a-1},\eta_a;\eta_{a+1}) \ot
I^\fm\big(0;,\{\eta_j,0\}_{j=1}^{k},\{\eta_j,0\}_{j=a+1}^n;1\big)  \\
=&\, \gd_{r=2a-2k-1,a>k}\,  \sum \eta_a  I^\fl(0;\{\eta_j,0\}_{j=k+1}^{a-1},\eta_a;1) \ot
I^\fm\big(0;\{\eta_j,0\}_{j=1}^{k},\{\eta_j,0\}_{j=a+1}^n;1\big)  \\
=&\, \gd_{r=2a-2k-1}\, M^\fl(2_{a-k-1}, \check{1}) \ot \tz^\fm(2_{n-a+k})\\
=&\, \left\{
  \begin{array}{ll}
     \in \calL_r\ot \calH_{w-r}  \text{ by Lemma~\ref{lem:LieUnram}} \qquad & \hbox{if $k<a-1$;} \\
     \text{it is canceled by $\ncic{8}$)} & \hbox{if $k=a-1>0$,}
  \end{array}
\right.\\
\ncic{6}=&\, \gd_{r=2a-2k-1,a>k}\,  \sum \eta_a\eta_{a+1}
I^\fl(0;\{\eta_j,0\}_{j=k+1}^{a-1},\eta_a,\{\eta_j,0\}_{j=a+1}^{b-1};\eta_b) \ot
I^\fm\big(0;\{\eta_j,0\}_{j=1}^{k},\{\eta_j,0\}_{j=b}^n;1\big)  \\
= &\,\gd_{r=2b-2k-1,b>a>k}\, M^\fl(2_{a-k-1}, \check{1},2_{b-a-1}) \ot \tz^\fm(2_{n-b+k+1}) \in\calL_r\ot \calH_{w-r}
\quad\text{by induction}, \\
\ncic{7}=&\,\gd_{r=2n-2k-1,a>k}\,  \sum \eta_a\eta_{a+1} I^\fl(0;\{\eta_j,0\}_{j=k+1}^{a-1},\eta_a,\{\eta_j,0\}_{j=a+1}^n;1)
\ot I^\fm\big(0;\{\eta_j,0\}_{j=1}^{k};1\big)  \\
= &\,\gd_{r=2n-2k-1,b>a>k}\, M^\fl(2_{a-k-1}, \check{1},2_{n-a}) \ot \tz^\fm(2_k) \in\calL_r\ot \calH_{w-r} \quad\text{by
induction}\\
&\hskip3cm (\text{this can also be regarded as the $b=n+1$ case of $\ncic{6}${}}), \\
\ncic{8}=&\,  \gd_{r=2b-2a-1,b>a}\,  \sum \eta_a\eta_{a+1} I^\fl(\eta_a;\{\eta_j,0\}_{j=a+1}^{b-1},\eta_b;0) \ot
I^\fm\big(0;\{\eta_j,0\}_{j=1}^{a},\{\eta_j,0\}_{j=b+1}^n;1\big)  \\
=&\, - \gd_{r=2b-2a-1,b>a}\,  \sum  \eta_{a+1} I^\fl(0;\eta_b,\{0,\eta_j\}_{j=b-1}^{a+1};1) \ot
I^\fm\big(0;\{\eta_j,0\}_{j=1}^{a},\{\eta_j,0\}_{j=b+1}^n;1\big)  \\
=&\, - \gd_{r=2b-2a-1,b>a}\,   M^\fl(2_{b-a-1}, \check{1})  \ot  \tz^\fm(2_{n-b+a})  \\
=&\, \left\{
  \begin{array}{ll}
    \in\calL_r\ot \calH_{w-r} \text{ by Lemmas~\ref{lem:LieUnram}}\quad & \hbox{if $b>a+1$;} \\
    \text{it is canceled by $\ncic{1}$ } & \hbox{if $a=1,b=a+1$;}\\
    \text{it is canceled by $\ncic{5}$ } & \hbox{if $a>1, b=a+1$,}
  \end{array}
\right. \\
\ncic{9}=&\,\gd_{r=2b-2a-3,b>a+1}\,  \sum \eta_a\eta_{a+1} I^\fl(\eta_{a+1};0,\{\eta_j,0\}_{j=a+2}^{b-1};\eta_b) \ot
I^\fm\big(0;\{\eta_j,0\}_{j=1}^{a+1},\{\eta_j,0\}_{j=b+1}^n;1\big)  \\
=&\, 0 \hskip1cm \text{(by Lemma~\ref{lem:D1S1bfl}\ref{Murakami-lemma9})} ,\\
\ncic{\hskip-1.2pt1\!0}=&\,\gd_{r=2n-2a-1}\,  \sum \eta_a\eta_{a+1} I^\fl(\eta_{a+1};0,\{\eta_j,0\}_{j=a+2}^n;1) \ot
I^\fm\big(0;\{\eta_j,0\}_{j=1}^{a+1};1\big)  \\
=&\, 0 \hskip1cm \text{(by Lemma~\ref{lem:D1S1bfl}\ref{Murakami-lemma9}, which is the $b=n+1$ case of $\ncic{9}$)}, \\
\ncic{\hskip-1.2pt1\!1}=&\, \ncic{\hskip-1.2pt1\!2}= 0 \hskip1cm \text{(by Lemma~\ref{lem:D1S1bfl}\ref{Murakami-lemma9})}.
\end{align*}
Further, from $\ncic{1}$, $\ncic{5}$\ and $\ncic{8}$\ we see that $D_1M^\fm(2_{a-1},\check{1},2_{n-a})=0$. Hence,
Theorem~\ref{thm:Glanois} implies that $M^\fm(2_{a-1},\check{1},2_{n-a})\in\calH_{2n-1}$. This completes the proof of the
theorem.
\end{proof}

\begin{thm}\label{thm:UnramMMVfamily2}
Suppose $n\ge 3$. Then the motivic MMVs  $M^\fm(\check{2}_{a-1},1,\check{2}_{n-a})$  are unramified for all $1<a<n$.
\end{thm}

\begin{proof}
We will prove the theorem by applying induction on $n$. By definition,
\begin{align*}
M^\fm(\check{2}_{a-1},1,\check{2}_{n-a})=&\, \sum_{\eta_j=\pm 1} \eta_1\eta_a\eta_{a+1} I^\fm(0;\eta_1,0,\dots, \eta_{a-1},0,
\eta_a,\eta_{a+1},0,\dots,\eta_n,0 ;1).
\end{align*}
When $n=3$, we have
\begin{align*}
M^\fm(\check{2},1,\check{2})=&\, T^\fm(2,1,2)\in \calH_5
\end{align*}
by \cite[Theorem~1.2(3)(iii)]{XuZhao2023Aug}.

Now we assume $n\ge 4$ and let $r$ be an odd integer. The following picture shows all the positive
ways to cut $M^\fm(\check{2}_{a-1},1,\check{2}_{n-a})$ when computing its image under the coaction $D_r$.
As before, let $w=2n-1$ be its weight.

\begin{center}
\begin{tikzpicture}[scale=0.9]
\node (A0) at (0.05,0)
{$0;\eta_1,0,\dots,\eta_k,0,\dots,\eta_{a-1},0,\eta_a,\eta_{a+1},0,\,\dots\,,\eta_b,0,\dots,\eta_c,0,\dots,\eta_n,0;1$};
\node (A1) at (-2.6,-0.4) {${}$};
\node (A2) at (-4.2,-0.9) {${}$};
\node (A3) at (-2.5,0.6) {${}$};
\node (A4) at (-0.6,0.3) {${}$};
\node (A5) at (-1.8,-0.8) {${}$};
\node (A6) at (0,0.5) {${}$};
\node (A7) at (3.95,-0.6) {${}$};
\node (A8) at (0.8,0.7) {${}$};
\node (A9) at (1.25,-0.5) {${}$};
\node (A10) at (3,-0.8) {${}$};
\node (A11) at (2.8,-0.4) {${}$};
\node (A12) at (4.1,0.5) {${}$};
\draw (-6.3,-0.25) to (-6.3,-0.4) to (A1) node {$\cic{1}$} to (-.1,-0.4) to (-.1,-0.25);
\draw (-6.33,-0.25) to (-6.33,-0.9) to (A2) node {$\cic{2}$} to (2.05,-0.9) to (2.05,-0.25);
\draw (-4.02,0.15) to (-4.02,0.6) to (A3) node {$\cic{3}$} to (-.9,0.6) to (-.9,0.15);
\draw (-3.95,0.15) to (-3.95,0.3) to (A4) node {$\cic{4}$} to (2.45,0.3) to (2.45,0.2);
\draw (-3.58,-0.25) to (-3.58,-0.8) to (A5) node {$\cic{5}$} to (0,-0.8) to (0,-0.25);
\draw (-3.58,0.25) to (-3.58,0.5) to (A6) node {$\cic{6}$} to (1.95,0.5) to (1.95,0.15);
\draw (-3.54,-0.25) to (-3.54,-0.6) to (A7) node {$\cic{7}$} to (6.4,-0.6) to (6.4,-0.25);
\draw (-.81,0.15) to (-.81,0.7) to (A8) node {$\cic{8}$} to (2.5,0.7) to (2.5,0.2);
\draw (.17,-0.25) to (.17,-0.5) to (A9) node {$\cic{9}$} to (1.95,-0.5) to (1.95,-0.25);
\draw (.1,-0.25) to (.1,-0.8) to (A10) node {$\cic{\hskip-1.2pt1\!0}$} to (6.5,-0.8) to (6.5,-0.25);
\draw (2.15,-0.25) to (2.15,-0.4) to (A11) node {$\cic{\hskip-1.2pt1\!1}$} to (3.65,-0.4) to (3.65,-0.25);
\draw (2,0.15) to (2,0.5) to (A12) node {$\cic{\hskip-1.2pt1\!2}$} to (6.4,0.5) to (6.4,0.25);
\node (A) at (0,-1.6) {Possible cuts of $D_r M^\fm(\check{2}_{a-1},1,\check{2}_{n-a})$ ($1< a< n$) for odd $r$.};
\end{tikzpicture}
\end{center}
The main difference between this family and that in Theorem~\ref{thm:UnramMMVfamily1} is that $a\ne 1$ here. Thus, we have
\begin{align*}
\ncic{1}=&\, \gd_{r=2a-1}\, M^\fl(\check{2}_{a-1},1) \ot t^\fm(2_{n-a})
   \in \calL_r\ot \calH_{w-r}  \text{ by Lemmas~\ref{lem:D1S1bfl}\ref{lem:MtVUnram} and \ref{lem:LieUnram}},\\
\ncic{2}=&\, \gd_{r=2b-3,b>a+1} \,  M^\fl(\check{2}_{a-1}, 1,\check{2}_{b-a-1}) \ot t^\fm(2_{n-b+1})\in\calL_r\ot \calH_{w-r}
\quad\text{by induction}, \\
\ncic{3}=&\, 0 \hskip1cm \text{(by Lemma~\ref{lem:D1S1bfl}\ref{Murakami-lemma9})}, \\
\ncic{4}= &\,-\gd_{r=2b-2k-1,b>a>k}\, M^\fl(2_{b-a-1},\check{1},2_{a-k}) \ot t^\fm(2_{n-b+k})\in\calL_r\ot \calH_{w-r} \quad\text{ by
Theorem~\ref{thm:UnramMMVfamily1}}, \\
\ncic{5}=&\, \left\{
  \begin{array}{ll}
     \in \calL_r\ot \calH_{w-r}  \text{ by Lemma~\ref{lem:LieUnram}} \qquad & \hbox{if $k<a-1$;} \\
     \text{it is canceled by $\ncic{8}$)} & \hbox{if $k=a-1>0$,}
  \end{array}
\right.\\
\ncic{6}= &\,\gd_{r=2b-2k-1,b>a>k}\, M^\fl(2_{a-k-1}, \check{1},2_{b-a-1}) \ot t^\fm(2_{n-b+k+1}) \in\calL_r\ot \calH_{w-r}
\quad\text{by Theorem~\ref{thm:UnramMMVfamily1}}, \\
\ncic{7}= &\,\gd_{r=2n-2k-1,b>a>k}\, M^\fl(2_{a-k-1}, \check{1},2_{n-a}) \ot t^\fm(2_k) \in\calL_r\ot \calH_{w-r} \quad\text{by
Theorem~\ref{thm:UnramMMVfamily1}}\\
&\hskip3cm (\text{this can also be regarded as the $b=n+1$ case of $\ncic{6}$}), \\
\ncic{8}=&\, \left\{
  \begin{array}{ll}
    \in\calL_r\ot \calH_{w-r} \text{ by Lemmas~\ref{lem:LieUnram}}\quad & \hbox{if $b>a+1$;} \\
    \text{it is canceled by $\ncic{5}$ } & \hbox{if $b=a+1$,}
  \end{array}
\right. \\
\ncic{9}=&\, 0 \hskip1cm \text{(by Lemma~\ref{lem:D1S1bfl}\ref{Murakami-lemma9})} ,\\
\ncic{\hskip-1.2pt1\!0}=&\, 0 \hskip1cm \text{by Lemma~\ref{lem:D1S1bfl}\ref{Murakami-lemma9}  (which is the $b=n+1$ case of $\ncic{9}${})}, \\
\ncic{\hskip-1.2pt1\!1}=&\, \ncic{\hskip-1.2pt1\!2}= 0 \hskip1cm \text{(by Lemma~\ref{lem:D1S1bfl}\ref{Murakami-lemma9})}.
\end{align*}
Further, from $\ncic{5}$\ and $\ncic{8}$\ we see that $D_1M^\fm(\check{2}_{a-1},1,\check{2}_{n-a})=0$. Hence,
Theorem~\ref{thm:Glanois} implies that $M^\fm(\check{2}_{a-1},1,\check{2}_{n-a})\in\calH_{2n-1}$. This completes the proof of the
theorem.
\end{proof}

\section{A height one family of unrmaified MMVs}\label{sec:htOneFamily}
Let $d$ be an even positive integer. Define a family of height one motivic MMVs by
\begin{align*}
V_\emptyset^\fm(d)= T^\fm(1_{d-1},2):=&\, \sum_{\eta_1,\dots,\eta_d=\pm 1}\left(\prod_{i=1}^d \eta_i \right)
I^\fm(0;\eta_1,\eta_2,\dots,\eta_d,0;1), \\
V_0^\fm(d)= M^\fm(1_{d-1},2;\{-1\}_{d-1},1):=&\, \sum_{\eta_1,\dots,\eta_d=\pm 1} \eta_1\eta_d
I^\fm(0;\eta_1,\eta_2,\dots,\eta_d,0;1),
\end{align*}
and for all $1\le j\le d$,
\begin{align*}
V_j^\fm(d)= M^\fm(1_{d-1},2;\bfgs_j):=&\, \sum_{\eta_1,\dots,\eta_d=\pm 1}  \eta_1\cdots\widehat{\eta_j}\cdots \eta_d
I^\fm(0;\eta_1,\eta_2,\dots,\eta_d,0;1),
\end{align*}
where $\widehat{\eta_j}$ means that $\eta_j$ is removed from the product, and
\begin{equation*}
  \bfgs_j=
\left\{
  \begin{array}{ll}
   \big( \{-1,1\}_k, \{1,-1\}_{d/2-k} \big), & \hbox{if $j=2k+1$;} \\
   \big( \{-1,1\}_{k-1},-1,-1, \{1,-1\}_{d/2-k}\big), \quad & \hbox{if $j=2k$.}
  \end{array}
\right.
\end{equation*}

In this section, we will prove that that $V_j(d)$ ($j=\emptyset, 0$ or $1< j<d$) are unramified by finding their explicit
expressions except for $V_0(d)$ for which the expression is only conjectural. Using this information we then lift these
expressions to the motivic version. For $V_0^\fm(d)$ we can actually prove it is unramified without assuming the conjectured
formula. In fact, the conjectured expression of $V_0(d)$ would imply the explicit expression of $V_0^\fm(d)$ in terms of
motivic MZVs.

\subsection{The height one families $V^\fm_j(d)$, $j\ne 0$}
We start with the exact expressions of their non-motivic version in terms of MZVs.

\begin{thm}\label{thm:UnramMMVfamily3ht1}
Let $d\in\N$ be even. Then $V_j(d)$  are all unramified for $j=\emptyset, 2,3,\dots,d-1$. More precisely, for all $1\le j\le
d$
\begin{align}\label{equ:Vjd-non-Motivic}
V_j(d)=&\,T(d+1)+2\zeta(\ol{d-j+1},\ol{j})-2\zeta(d-j+1,\ol{j}).
\end{align}
Further, $V_j(d)+2(\delta_{j=1}-\delta_{j=d}) T(d) \log 2$ is unramified for all $1\le j\le d.$
\end{thm}
\begin{proof}
By duality of MTVs, we immediately see that $V_\emptyset(d)=T(d+1)$ which is unramified.
Next, for all $1\le j\le d$, we have
\begin{equation*}
V_j(d)=\int_0^1\om_{-1}^{j-1} \om_1 \om_{-1}^{d-j} \om_0
\end{equation*}
where all the 1-forms $\om_{-1}, \om_1$ and $\om_0$ are defined by \eqref{defn:oms}. By duality,
\begin{align*}
V_j(d)=&\,\int_0^1   \om_{-1} \om_0^{d-j}(\om_0+\om_1-\om_{-1})\om_0^{j-1}\\
=&\,\int_0^1  (\tb-\tc)\ta^{d-j}(\ta+2\tc)\ta^{j-1}\\
=&\,T(d+1)+2\zeta(\ol{d-j+1},\ol{j})-2\zeta(d-j+1,\ol{j})
\end{align*}
which is unramified by the parity reduction formula \cite[Prop.~4.1]{XuZhao2023Aug}.
More explicitly,
\begin{align}
2\zeta^\fm(\ol{d+1-j},\ol{j})=&\, -\zeta^\fm(d+1)-(-1)^j 2\sum_{s=0}^{d/2}
\bigg[\binom{d-2s}{d-j} \zeta^\fm(\ol{d+1-2s})\zeta^\fm(2s) \notag\\
&\, +\binom{d-2s}{j-1} \zeta^\fm(\ol{d+1-2s})\zeta^\fm(2s)\bigg] +\delta_{2|j}
2\zeta^\fm(\ol{d+1-j})\zeta^\fm(\ol{j}),\label{equ:parityRed1}\\
2\zeta^\fm(d+1-j,\ol{j})=&\, -\zeta^\fm(\ol{d+1})-(-1)^j 2\sum_{s=0}^{d/2}
\bigg[\binom{d-2s}{d-j} \zeta^\fm(d+1-2s)\zeta^\fm(\ol{2s}) \notag\\
&\, +\binom{d-2s}{j-1} \zeta^\fm(\ol{d+1-2s}) \zeta^\fm(\ol{2s})\bigg] +\delta_{2|j}
2\zeta^\fm(d+1-j)\zeta^\fm(\ol{j}).\label{equ:parityRed2}
\end{align}
Observe that $\zeta^\fm(\ol{1})$ can appear only when $d-2s=0$ and,  either $j=1$ or $j=d$. If $j=1$ then the ramified parts
of $V_1(d)$ combine to yield
$$
2\zeta^\fm(\ol{1})\big(\zeta^\fm(d)-\zeta^\fm(\ol{d})\big)
=-2T^\fm(d)\log^\fm(2).
$$
If $j=d$ then $\delta_{2|j}=1$ and the ramified parts of $V_d(d)$ combine to yield
$$
2\zeta^\fm(\ol{1})\big(\zeta^\fm(\ol{d})-\zeta^\fm(d)\big)
=2T^\fm(d)\log^\fm(2).
$$
This completes the proof of the theorem after we apply the period map.
\end{proof}

\begin{cor}\label{cor:DrOnDblEulerSums}
Let $d\in\N$ be even and $1\le j\le d$. Then we have
\begin{align*}
  D_1\Big(\zeta^\fm(\ol{d+1-j},\ol{j})-\zeta^\fm(d+1-j,\ol{j})\Big)
=&\,
\Big(\delta_{j=1}-\delta_{j=d}\Big)  \zeta^\fl(\ol{1})\ot T^\fm(d) .
\end{align*}
For any positive odd integer $1<r<d$,
\begin{align*}
&\, D_r\Big(\zeta^\fm(\ol{d+1-j},\ol{j})-\zeta^\fm(d+1-j,\ol{j})\Big)\\
=&\,  (-1)^j
\bigg[\binom{r-1}{d-j} \Big(\zeta^\fl(r)\ot \zeta^\fm(\ol{d+1-r})
-\zeta^\fl(\ol{r})\ot \zeta^\fm(d+1-r)\Big)  \\
&\,\qquad  +\binom{r-1}{j-1}  \Big(\zeta^\fl(\ol{r})\ot \zeta^\fm(\ol{d+1-r})
- \zeta^\fl(\ol{r})\ot \zeta^\fm(d+1-r) \Big) \bigg] .
\end{align*}
Moreover,
\begin{align}\label{equ:diffInLie}
2\zeta^\fl(\ol{d+1-j},\ol{j})-2\zeta^\fl(d+1-j,\ol{j})=&\, -T^\fl(\ol{d+1})-(-1)^j  \binom{d}{j} T^\fl(d+1).
\end{align}

\end{cor}
\begin{proof}
The formula for $D_r$ follows from  \eqref{equ:parityRed1} and \eqref{equ:parityRed2} using
its derivation property $D_r(XY)=(1\ot X)D_r(Y)+(1\ot Y)D_r(X)$ and the fact that $\zeta^\fl(n)=\zeta^\fl(\ol{n})=0$ for all
even $n$. The second claim
is also clear from \eqref{equ:parityRed1} and \eqref{equ:parityRed2} since modulo products we now have
\begin{align*}
2\zeta^\fl(\ol{d+1-j},\ol{j})=&\, -\zeta^\fl(d+1)+(-1)^j
\bigg[\binom{d}{j}+\binom{d}{j-1}\bigg]\zeta^\fl(\ol{d+1})  , \\
2\zeta^\fl(d+1-j,\ol{j})=&\, -\zeta^\fl(\ol{d+1})+(-1)^j
\bigg[\binom{d}{j} \zeta^\fl(d+1)+\binom{d}{j-1}\zeta^\fl(\ol{d+1}) \bigg] .
\end{align*}
Thus, the definition $T^\fl(d+1)=\zeta^\fl(d+1)-\zeta^\fl(\ol{d+1})$ yields \eqref{equ:diffInLie} immediately.
\end{proof}

We now provide two simple lemmas which will be used in this paper and our subsequent work on unramified alternating MMVs.

\begin{lem}\label{lem:T1n}
For all positive integer $n$,
\begin{align*}
T^\fm(1_n)=\sum \eta_1\cdots \eta_n I^\fm\big(0;\eta_1,\dots,\eta_n;1\big)= \frac{(\log^\fm 2)^n}{n!} .
\end{align*}
\end{lem}

\begin{proof}
Let $\eps$ be a small positive number. Then, by duality substitution $t\to (1-t)/(1+t)$, which is motivic by
\cite[Appendix A.3]{Glanois2015} we see that
\begin{align*}
 T_\eps(1_n):=&\, \int_0^{1-\eps} \om_{-1}^n= \int_{\eps/(2-\eps)}^1 \om_0^n
 =\frac1{n!} \left(\int_{\eps/(2-\eps)}^1 \om_0\right)^n
 =\frac1{n!}\left(-\log \frac{\eps}{2-\eps}\right)^n \\
 =&\, \frac1{n!}\left(-\log \eps+\log 2 \right)^n +O(\eps\log^{n-1} \eps).
\end{align*}
We see that the regularized value
\begin{equation}\label{equ:Treg1n}
\dch(T^\fm(1_n))= T^{\reg}(1_n) =\frac1{n!}(\log 2 )^n.
\end{equation}

We now prove the lemma by induction on $n$. If $n=1$, this is trivial. Now we assume $n>1$
and consider all the possible cuts of $D_r T^\fm(1_n)$ as shown in the following picture:
\begin{center}
\begin{tikzpicture}[scale=0.9]
\node (A0) at (0.05,0) {$0;\eta_1,\ldots\, ,\eta_k, \ldots\, ,\eta_{r+1},\ldots\, ,\eta_{n-r},\ldots\, ,  \eta_{k+r+1},
\dots,\eta_n;1$};
\node (A1) at (-3.4,-0.4) {${}$};
\node (A2) at (-0.3,0.4) {${}$};
\node (A3) at (2.7,-0.4) {${}$};
\draw (-4.86,-0.25) to (-4.86,-0.4) to (A1) node {$\cic{1}$} to (-1.2,-0.4) to (-1.2,-0.25);
\draw (-3.0,0.15) to (-3.0,0.4) to (A2) node {$\cic{2}$}to  (2.6,0.4) to (2.6,0.15);
\draw (.6,-0.25) to (.6,-0.4)   to (A3) node {$\cic{3}$}to (4.9,-0.4) to (4.9,-0.25);
\node (A) at (0,-1.6) {Possible cuts of $D_r T^\fm(1_n)$ for odd $r$.};
\end{tikzpicture}
\end{center}

Then we have
\begin{align*}
\ncic{1}=&\,  \sum \eta_1\cdots \eta_n I^\fl\big(0;\eta_1,\dots,\eta_r;\eta_{r+1}\big)\ot
I^\fm\big(0;\{\eta_i\}_{i=r+1}^n;1\big) \\
=&\,   T^\fl(1_r) \ot  X^\fm_1(n-r)
\end{align*}
where we have set
\begin{equation*}
X^\fm_j(d)=\sum \eta_1\cdots \widehat{\eta_j} \cdots \eta_n I^\fm\big(0;\eta_1,\dots,\eta_n;1\big).
\end{equation*}
Similarly,
\begin{align}
\ncic{2}=&\,\sum_{k=2}^{n-r-1}  \sum \eta_1\cdots \widehat{\eta_{k+r+1}}\cdots\eta_n I^\fl\big(0;\eta_{k+1},\dots,\eta_{k+r};1\big)\ot
I^\fm\big(0;\{\eta_i\}_{i=1}^k,\{\eta_i\}_{i=k+r+2}^n;1\big) \label{equ:showCancelEta1}\\
-&\,\sum_{k=2}^{n-r-1}  \sum \eta_1 \cdots \widehat{\eta_{k}}\cdots\eta_n I^\fl\big(0;\eta_{k+1},\dots,\eta_{k+r};1\big)\ot
I^\fm\big(0;\{\eta_i\}_{i=1}^k,\{\eta_i\}_{i=k+r+2}^n;1\big) \label{equ:showCancelEta2}\\
=&\,\sum_{k=2}^{n-r-1}   T^\fl(1_r) \ot \Big( X^\fm_{k+1}(n-r)-X^\fm_{k}(n-r)\Big) \notag
\end{align}
by path reversal, namely, $I^\fl(0;a_1,\dots,a_r;1)=I^\fl(0;a_r,\dots,a_1;1)$ for all odd $r$. Also notice that when we
changed the variables $\eta_{k+1}\to \eta_{k+1}\eta_{k},\dots,\eta_{k+r}\to \eta_{k+r}\eta_{k}$ to produce
\eqref{equ:showCancelEta2}  the extra factor $\eta_{k}^r=\eta_{k}$ appears in the front which cancels the original $\eta_{k}$.
The same is true when producing \eqref{equ:showCancelEta1}.

Similar computation shows that
\begin{align*}
\ncic{3}=&\,  T^\fl(1_r) \ot \Big(T^\fm(1_{n-r})-X^\fm_{n-r}(n-r)\Big) .
\end{align*}

Putting all the above together, by Lemma~\ref{lem:T1n} for $r>1$ and by telescoping for $r=1$, we see that
\begin{align*}
D_r T^\fm(1_n)
=\left\{
   \begin{array}{ll}
 \log^\fl 2\ot  T^\fm_1(n-1), & \hbox{if $r=1$;} \\
    0, & \hbox{if $r> 1$.}
   \end{array}
 \right.
\end{align*}
It is clear that $D_r(\log^\fm 2)^n=0$ except for  $D_1(\log^\fm 2)^n=n\log^\fl 2 \ot (\log^\fm 2)^{n-1}$.
By the descent criterion in Theorem~\ref{thm:Glanois}, we see that
$T^\fm(1_n)-(\log^\fm 2)^n/n!=c \zeta^\fm (n)$ for some $c\in\Q$. Applying the period map $\dch$ and
using \eqref{equ:Treg1n} we see that $c=0$. This concludes our proof of the lemma.
\end{proof}

\begin{lem}\label{lem:Xjn}
For all positive integers $1\le j\le n$, we have
\begin{align}\label{equ:DiffXT}
X^\fm_j(n)=T^\fm(1_n)+\sum_{s=0}^{n-j} (-1)^{s+n-j} 2 \binom{n-1-s}{j-1}\zeta^\fm (\ol{n-s})T^\fm(1_{s}).
\end{align}
In particular,
\begin{equation*}
 X^\fl_j(n)=
 \left\{
   \begin{array}{ll}
      (-1)^{n-j} 2\binom{n-1}{j-1}\zeta^\fl (\ol{n}), & \hbox{if $n\ge 1$;} \\
    -\log^\fl 2, & \hbox{if $j=n=1$.}
   \end{array}
 \right.
\end{equation*}
\end{lem}
\begin{proof}
The proof goes in exactly the same way as that of Lemma~\ref{lem:T1n}. We first show that
\begin{align}\label{equ:DiffXReg}
X^\reg_j(n)=\frac{1}{n!}\log^n 2+\sum_{s=0}^{n-j} (-1)^{s+n-j} 2 \binom{n-1-s}{j-1}\zeta (\ol{n-s}) \frac{(\log 2)^{s}}{s!}.
\end{align}
Let $\eps$ be a small positive number. Then, by duality substitution $t\to (1-t)/(1+t)$
\begin{align*}
a_{j,n}^{(\eps)}:=\int_0^{1-\eps} \om_{-1}^{j-1}\om_1\om_{-1}^{n-j}
=&\,  \int_{\eps/(2-\eps)}^1 \om_0^{n-j}(\om_0+2\tc)\om_0^{j-1} \\
=&\,  \int_{\eps/(2-\eps)}^1 \om_0^n+ 2\int_{\eps/(2-\eps)}^1 \om_0^{n-j} \tc \om_0^{j-1}\\
=&\,\frac{1}{n!}(\log 2-\log \eps)^n+ 2\int_{\eps/(2-\eps)}^1 \om_0^{n-j} \tc \om_0^{j-1}+O(\eps \log^{n-1} \eps).
\end{align*}
This readily yields \eqref{equ:DiffXReg} when $j=n$, namely, $a_{n,n}(T)=\frac{1}{n!}(\log 2+T)^n+ 2 \zeta^\fl (\ol{n})$
yielding $a_{n,n}^\reg=\frac{1}{n!}\log^n 2+ 2 \zeta (\ol{n})$.

Assume $j\ge 2$. Then, by the shuffle product property of iterated integrals and an induction on $n$, we obtain
\begin{align*}
(j-1) a_{j,n}=&\,
\int_0^{1-\eps}\om_{-1} \int_0^{1-\eps}\om_{-1}^{j-2}\om_1\om_{-1}^{n-j}
 - (n-j+1) \int_0^{1-\eps}\om_{-1}^{j-2}\om_1\om_{-1}^{n-j+1}  \\
 =&\, a_{j-1,n-1}( \log 2-\log \eps) -(n-j+1) a_{j-1,n}+O(\eps \log^{n-1} \eps).
\end{align*}
Hence, the regularization of $a_{j,n}$'s satisfy the recursive relation
\begin{align*}
(j-1)a_{j,n}^\reg=&\,a_{j-1,n-1}\log 2  -(n-j+1) a_{j-1,n}^\reg.
\end{align*}
This leads to \eqref{equ:DiffXReg} by an downward induction starting with  $a_{n,n}^\reg=\frac{1}{n!}\log^n 2+ 2 \zeta
(\ol{n})$.

We now denote by $R^\fm_j(n)$ the right-hand side of \eqref{equ:DiffXT}. Then it is easy to see that for all odd $r>1$ and
$r<n$
\begin{equation}\label{equ:DrRfm1}
D_r R^\fm_j(n)=
X^\fl_j(r) \ot T^\fm(1_{n-r}).
\end{equation}
Moreover, easy computation leads to
\begin{align}
D_1 R^\fm_j(n)=&\,  -2\delta_{j=1} \log^\fl 2 \ot T^\fm(1_{n-1})+T^\fl(1)\ot X^\fm_j(n-1). \label{equ:DrRfm2}
\end{align}

Now we prove \eqref{equ:DiffXT} by induction on $n$. If $n=1$ then $j=1$ and \eqref{equ:DiffXT} is clear since
$\zeta^\fm(\ol{1})=-\log^\fm 2=X^\fm_1(1)=I^\fm(0;-1;1)$.
We assume $n\ge 2$, \eqref{equ:DiffXT} holds if we replace $n$ by any smaller number. Let's consider all the possible cuts of
$D_rX^\fm_j(n)$ using the same picture in the proof of Lemma~\ref{lem:T1n} and show  $D_r\big(X^\fm_j(n)-R^\fm(1_n) \big)=0$.
We need to consider whether $\eta_j$ lies within the quotient sequence in the first factor of the cuts in the proceeding
picture. Hence, by setting
\begin{align*}
X_{a,b}^\fm(n):=&\, \sum \eta_1\cdots\widehat{\eta_a}\cdots\widehat{\eta_b}\cdots\eta_n I^\fm\big(0;\{\eta_i\}_{i=1}^n;1\big)
\end{align*}
for all $1\le a<b\le n$, we have
\begin{align*}
\ncic{1}{}\text{(i)}=&\, \gd_{j\le r}   X_j^\fl(r) \ot T^\fm(1_{n-r}),\\
\ncic{1}{}\text{(ii)}=&\,   \gd_{j= r+1} T^\fl(1_r) \ot T^\fm(1_{n-r}),\\
\ncic{1}{}\text{(iii)}=&\,  \gd_{j\ge r+2}  T^\fl(1_r) \ot X_{1,j-r}^\fm(n-r) ,\\
\ncic{2}{}\text{(i)}=&\, \gd_{j<n-r-1} T^\fl(1_r) \ot \Big( X^\fm_{j,n-r}(n-r)-X^\fm_{j,j+1}(n-r)\Big),\\
\ncic{2}{}\text{(ii)}=&\, \gd_{j<n-r}  T^\fl(1_r) \ot \Big( X^\fm_{j,j+1}(n-r)-T^\fm(1_{n-r})\Big),\\
\ncic{2}{}\text{(iii)}=&\,0 \quad   (\text{for $k<j\le k+r$}),\\
\ncic{2}{}\text{(iv)}=&\, \gd_{j\ge r+2} T^\fl(1_r) \ot \Big(T^\fm(1_{n-r})- X^\fm_{j-r-1,j-r}(n-r)\Big),\\
\ncic{2}{}\text{(v)}=&\, \gd_{j>r+2}   T^\fl(1_r) \ot \Big( X_{j-r-1,j-r}^\fm(n-r)- X_{1,j-r}^\fm(n-r)\Big),\\
\ncic{3}{}\text{(i)}=&\, \gd_{j<n-r}  T^\fl(1_r) \ot \Big( X^\fm_j(n-r)-X^\fm_{j,n-r}(n-r)\Big),\\
\ncic{3}{}\text{(ii)}=&\,  \gd_{j=n-r} T^\fl(1_r) \ot \Big( X^\fm_{n-r}(n-r)-T^\fm(1_{n-r})\Big),\\
\ncic{3}{}\text{(iii)}=&\, 0 \quad  (\text{for $n-r<j$}).
\end{align*}
Here, we can divide the above 11 terms in three groups corresponding to the three different ways of cutting depending on the position of $j$. Replacing $\gd_{j>r+2}$ by $\gd_{j\ge r+2}$ in $\ncic{2}{}\text{(v)}${}\!, we finally
obtain
\begin{align*}
D_r X^\fm_j(n)
=&\, \left\{
   \begin{array}{ll}
  \gd_{j=1} (X_1^\fl(1)-T^\fl(1)) \ot T^\fm(1_{n-1})+T^\fl(1) \ot  X^\fm_j(n-1) , & \hbox{if $r=1$;} \\
   \gd_{j\le r } X_j^\fl(r) \ot T^\fm(1_{n-r}) , & \hbox{if $r> 1$,}
   \end{array}
 \right. \\
 =&\, D_r R^\fm_j(n)
\end{align*}
by \eqref{equ:DrRfm1} and \eqref{equ:DrRfm2}. By Theorem~\ref{thm:Glanois}, our theorem is now proved.
\end{proof}

We now lift \eqref{equ:Vjd-non-Motivic} to the motivic version by applying the descent theory.

\begin{thm}\label{thm:UnramMMVfamily3}
Let $d\in\N$ be even. Then $V^\fm_j(d)$  are all unramified for $j=\emptyset, 2,3,\dots,d-1$. More precisely, for all $1\le
j\le d$
\begin{align}\label{equ:Vjd-Motivic}
V_j^\fm(d)=&\,T^\fm(d+1)+2\zeta^\fm(\ol{d-j+1},\ol{j})-2\zeta^\fm(d-j+1,\ol{j}).
\end{align}
Further, $V^\fm_j(d)+2(\delta_{j=1}-\delta_{j=d}) T^\fm(d) \log^\fm 2$ is unramified for all $1\le j\le d.$
\end{thm}
\begin{proof}
We must verify the images of the coactions $D_r$ on both sides of \eqref{equ:Vjd-Motivic} are the same for all positive odd
$r<d$. Hence, we need to consider the following cuts for $V^\fm_j(d)$.

\begin{center}
\begin{tikzpicture}[scale=0.9]
\node (A0) at (0.05,0) {$0;\eta_1,\ldots\, , \ldots\, ,\eta_r,\eta_{r+1},\dots,\eta_k, \ldots\, , \ldots\, ,\eta_{k+r+1},
\dots,\eta_d,0;1$};
\node (A11) at (-3.5,-0.6) {$\cic{1}\cic{2}\cic{3}$};
\node (A22) at (1.3,0.45) {$\cic{4}\cic{5}\cic{6}\cic{7}\cic{8}$};
\node (A33) at (2,-0.55) {$\cic{9}\cic{\hskip-1.2pt1\!0}\cic{\hskip-1.2pt1\!1}$};
\node (A44) at (2.8,-0.95)  {$\cic{\hskip-1.2pt1\!2}\cic{\hskip-1.3pt1\!3}$};
\draw (-5.1,-0.15) to (-5.1,-0.3) to  (-1.6,-0.3) to (-1.6,-0.15);
\draw (0,0.15) to (0,0.3) to  (2.55,0.3) to (2.55,0.15);
\draw (0,-0.15) to (0,-0.3) to  (4.95,-0.3) to (4.95,-0.15);
\draw (-0.07,-0.15) to (-0.07,-0.725)   to (5.2,-0.725) to (5.2,-0.15);
\node (A) at (0,-1.6) {Possible cuts of $D_r V^\fm_j(d)$ ($1\le j\le d$) for odd $r$.};
\end{tikzpicture}
\end{center}
Then,
\begin{align*}
\ncic{1}=&\, \gd_{r\ge j}\, \sum \eta_1\cdots \widehat{\eta_j}\cdots\eta_d
I^\fl\big(0;\eta_1,\dots,\eta_j,\dots,\eta_r;\eta_{r+1}\big)\ot I^\fm\big(0;\{\eta_i\}_{i=r+1}^d,0;1\big) \\
=&\, \gd_{r\ge j}\, \sum \eta_1\cdots \widehat{\eta_j}\cdots\eta_r I^\fl\big(0;\eta_1,\dots,\eta_j,\dots,\eta_r;1\big)\ot
\eta_{r+1}\cdots \eta_d I^\fm\big(0;\{\eta_i\}_{i=r+1}^d,0;1\big) \\
:=&\, \gd_{r\ge j}\, X^\fl_j(r) \ot V_\emptyset^\fm(d-r)
=\left\{
   \begin{array}{ll}
     (-1)^j 2\binom{r-1}{j-1}\zeta^\fl(\ol{r}) \ot T^\fm(d-r), & \hbox{if $r>1$;} \\
    -\log^\fl 2\ot T^\fm(d), & \hbox{if $r=j=1$,}
   \end{array}
 \right.
\end{align*}
by Lemma~\eqref{lem:Xjn}. In $\ncic{1}${}\! , when we changed the variables $\eta_1\to \eta_1\eta_{r+1},\dots,\eta_r\to
\eta_r\eta_{r+1}$ we only produced $\eta_{r+1}^{r-1}=1$ as the extra factor in the front because $\eta_j$ is missing. But for
$\ncic{2}${}\! next, we actually produced
the extra factor $\eta_{r+1}^r=\eta_j^r=\eta_j$:
\begin{align*}
\ncic{2}=&\,\gd_{r=j-1} \sum \eta_1\cdots \widehat{\eta_j}\cdots\eta_d I^\fl\big(0;\eta_1,\dots,\eta_{j-1};\eta_j\big)\ot
I^\fm\big(0;\{\eta_j\}_{i=j}^d,0;1\big) \\
=&\, \gd_{r=j-1}\, \sum \eta_1\cdots \eta_{j-1} I^\fl\big(0;\eta_1,\dots,\eta_{j-1};\eta_j\big)\ot \eta_j\cdots \eta_d
I^\fm\big(0;\{\eta_i\}_{i=j}^d,0;1\big) \\
=&\, \gd_{r=j-1}\, T^\fl(1_r) \ot V^\fm_\emptyset(d-r)
=\left\{
   \begin{array}{ll}
  0, & \hbox{if $r>1$;} \\
  \log^\fl 2\ot T^\fm(d), & \hbox{if $r=1,j=2$.}
   \end{array}
 \right.
\end{align*}

Similarly,
\begin{align*}
\ncic{3}=&\, \gd_{r\le j-2}\,  \sum \eta_1\cdots \eta_{r+1}\cdots\widehat{\eta_j}\cdots\eta_d
I^\fl\big(0;\eta_1,\dots,\eta_r;\eta_{r+1}\big)\ot I^\fm\big(0;\{\eta_i\}_{i=r+1}^d,0;1\big) \\
=&\, \gd_{r\le j-2}\, \sum \eta_1\cdots \eta_r I^\fl\big(0;\eta_1,\dots,\eta_r;1\big)\ot
\eta_{r+2}\cdots\widehat{\eta_j}\cdots \eta_d I^\fm\big(0;\{\eta_i\}_{i=r+1}^d,0;1\big) \\
:=&\, \gd_{r\le j-2}\, T^\fl (1_r) \ot V_{1,j-r}^\fm(d-r)
=\left\{
   \begin{array}{ll}
  0, & \hbox{if $r>1$;} \\
  \log^\fl 2\ot  V_{1,j-1}^\fm(d-1), & \hbox{if $r=1,j\ge 3$,}
   \end{array}
 \right.
\end{align*}
where, for all $1\le a<b\le n$ we set
\begin{align*}
V_{a,b}^\fm(n):=&\, \sum \eta_1\cdots\widehat{\eta_a}\cdots\widehat{\eta_b}\cdots\eta_n
I^\fm\big(0;\{\eta_i\}_{i=1}^n,0;1\big).
\end{align*}
Next, taking $j<k$, $k=j$ and so on in the cutting picture, we get using
Cor.~\ref{equ:diffInLie}, Lemma~\ref{lem:Xjn}  and the induction assumption:
\begin{align*}
\ncic{4}=&\,\left\{
   \begin{array}{ll}
  0, & \hbox{if $r>1$;} \\
  \log^\fl 2\ot  \Big(V^\fm_{j,d-1}(d-1)-V^\fm_{j,j+1}(d-1)\Big)  , & \hbox{if $r=1,j\le d-2$,}
   \end{array}
 \right.\\
\ncic{5}=&\,\left\{
   \begin{array}{ll}
  0, & \hbox{if $r>1$;} \\
  \log^\fl 2\ot  \Big(V^\fm_{j,j+1}(d-1)-T^\fm(d)\Big), & \hbox{if $r=1,j\le d-2$,}
   \end{array}
 \right.\\
\ncic{6}=&\, 0 \quad (\text{if $r\ge j-k$}),\\
\ncic{7}=&\,\left\{
   \begin{array}{ll}
  0, & \hbox{if $r>1$;} \\
  \log^\fl 2\ot  \Big( T^\fm(d)-V^\fm_{j-2,j-1}(d-1)\Big), & \hbox{if $r=1$,}
   \end{array}
 \right.\\
\ncic{8}=&\,\left\{
   \begin{array}{ll}
  0, & \hbox{if $r>1$;} \\
  \log^\fl 2\ot \Big(V^\fm_{j-2,j-1}(d-1)-V^\fm_{1,j-1}(d-1)\Big)   , & \hbox{if $r=1,j\ge 4$,}
   \end{array}
 \right.\\
\ncic{9}=&\,\left\{
   \begin{array}{ll}
  0, & \hbox{if $r>1$;} \\
  -\log^\fl 2\ot  T^\fm(d), & \hbox{if $r=1,j=d-1$,}
   \end{array}
 \right.\\
\ncic{\hskip-1.2pt1\!0}=&\, \left\{
   \begin{array}{ll}
  - (-1)^j  2\binom{r-1}{d-j} \zeta(\ol{r}) \ot T^\fm(d-r) , & \hbox{if $r>1$;} \\
  \log^\fl 2\ot  T^\fm(d), & \hbox{if $r=1,j=d$,}
   \end{array}
 \right.\\
\ncic{\hskip-1.2pt1\!1}=&\,\left\{
   \begin{array}{ll}
   0, & \hbox{if $r>1$;} \\
  -\log^\fl 2\ot  V^\fm_{j,d-1}(d-1), & \hbox{if $r=1,j\le d-2$,}
   \end{array}
 \right. \\
\ncic{\hskip-1.2pt1\!2}=&\,0  \quad (\text{if $r\le d+1-j$}),\\
\ncic{\hskip-1.2pt1\!3}=&\,\left\{
   \begin{array}{ll}
   -\gd_{r>  d+1-j}\,(-1)^j  \binom{r-1}{d-j} T^\fl(r)\ot \big(-2 \zeta^\fm(\ol{d-r+1})\big), & \hbox{if $r>1$;} \\
  0, & \hbox{if $r=1$.}
   \end{array}
 \right.
\end{align*}

If $r>1$ then putting all the nonzero cuts $\ncic{1}${}\! , $\ncic{8}${}\ and $\ncic{\hskip-1.2pt1\!1}${}\ together, we arrive
at
\begin{align*}
D_r V_j^\fm(d)=&\,  (-1)^j 2\binom{r-1}{j-1}\zeta^\fl(\ol{r}) \ot T^\fm(d-r)
   - (-1)^j  2\binom{r-1}{d-j} \zeta(\ol{r}) \ot T^\fm(d-r)  \\
  &\,  +(-1)^j 2 \binom{r-1}{d-j} T^\fl(r)\ot   \zeta^\fm(\ol{d-r+1})\\
  =&\,   D_r \Big(T^\fm(d+1)+2\zeta^\fm(\ol{d-j+1},\ol{j})-2\zeta^\fm(d-j+1,\ol{j}) \Big)
\end{align*}
by Cor.~\ref{cor:DrOnDblEulerSums}. If $r=1$ then after many cancellations we see that
\begin{align*}
D_1 V_j^\fm(d)=&\,2\Big(\delta_{j=d}- \delta_{j=1}\Big)\log^\fl 2\ot T^\fm(d)\\
=&\, D_1 \Big(T^\fm(d+1)+2\zeta^\fm(\ol{d-j+1},\ol{j})-2\zeta^\fm(d-j+1,\ol{j}) \Big).
\end{align*}
This completes the proof of the theorem by the descent criterion in Theorem~\ref{thm:Glanois} in view of
Theorem~\ref{thm:UnramMMVfamily3ht1}.
\end{proof}

\begin{exa}
If we use the parity reduction for double Euler sums (cf. \cite[Prop.~4.1]{XuZhao2023Aug}) we can derive the exact expression
of $V_j(d)$ ($1\le j<d$) in terms of (products of) Riemann zeta values precisely. For example, we have
\begin{equation*}
64V_5(8)=64 M(1_7,2; \{-1,1\}_2,\{1,-1\}_2)=7154\zeta(9)-4075\zeta(2)\zeta(7)- 257 \zeta(4)\zeta(5) .
\end{equation*}
\end{exa}

\subsection{The height one family $V^\fm_0(d)$}
We first remark that from numerical computation the following conjecture must be true.

\begin{conj}\label{conj:EulerSharp212bar}
For all positive integers $d\ge 3$ we have
\begin{equation*}
    \zeta^\sharp(2, 1_{d-3}, \ol{2})=  \frac12 \zeta(d+1)-2\sum_{m+n=d+1,2|n} \zeta(m,n).
\end{equation*}
\end{conj}

\begin{thm}\label{thm:assumConj}
Assuming Conjecture~\ref{conj:EulerSharp212bar}, we have
\begin{equation}\label{equ:V0d}
V_0(d)= \frac12\zeta^\star(1,d)- \frac{1}4 \zeta(d+1)+\sum_{m+n=d+1,2|n} \zeta(m,n).
\end{equation}
\end{thm}

\begin{proof}
By definition,
\begin{equation}\label{equ:Vdd}
V_0(d)=\int_0^1\om_{-1} \om_1^{d-2} \om_{-1} \om_0=\int_0^1 (\tb-\tc) \ta (\ta+2\tc)^{d-2}\ta
\end{equation}
by duality, where
\begin{equation*}
  \ta=\frac{dt}{t},\quad \tb=\frac{dt}{1-t},\quad \ta=\frac{dt}{-1-t}.
\end{equation*}
If $d=2$, we see easily that $V_2(2)=\zeta(3)-\zeta(\ol{3})=\frac{7}4 \zeta(3)$ while
the right-hand side of \eqref{equ:Vdd} is
\begin{equation*}
\frac{1}4 \zeta(3)+\frac32\zeta(1,2)=\frac{7}4 \zeta(3).
\end{equation*}
Thus the theorem holds when $d=2.$

For general $d\ge 3,$ in terms of Glanois's Euler $\sharp$-sums defined in \cite{Glanois2015}, the right-hand side of
\eqref{equ:V0d} can be written as
\begin{equation*}
\frac12\Big(\zeta^\sharp(\ol{2}, 1_{d-3}, \ol{2})-\zeta^\sharp(2, 1_{d-3}, \ol{2})\Big).
\end{equation*}
But by \cite[Theorem~4.8]{LinebargerZh2015}, when $d\ge 3$
\begin{equation*}
    \zeta^\sharp(\ol{2}, 1_{d-3}, \ol{2})=\zeta^\star(1,d).
\end{equation*}
Hence, \eqref{equ:V0d} holds if we assume Conjecture~\ref{conj:EulerSharp212bar}.
\end{proof}

We can now lift Theorem~\ref{thm:assumConj} to is motivic version.
First, we need some preliminary lemmas.

\begin{lem}\label{lem:t111Unram}
For all odd positive integers $r\ge 3$,
$$t^\fl(1_{r})=\frac{-2^{r-1}}{r} t^\fl(r) \quad\text{and}\quad t_1^\fl(1_{r-1})=2^{r-2} t^\fl(r).$$
In particular,  $t^\fl(1_r)=(1/2-2^{r-1})\zeta^\fl(r)\in \calL_r$ and therefore $t_1^\fl(1_{r-1})\in \calL_r$.
\end{lem}
\begin{proof}
For any positive integer $k$, by the stuffle relation for multiple $t$-values, we clearly have
\begin{align*}
t_*(k)t_*(1_{r-1})=\sum_{i=1}^{r-1} t_*(1_{i-1},k,1_{r-i-1})
+\sum_{i=1}^{r-2} 2 t_*(1_{i-1},k+1,1_{r-i-2}).
\end{align*}
Since stuffle relations are motivic (see \cite{Souderes2010}), modulo products we get
\begin{align}\label{equ:t1rReduce}
\sum_{i=1}^{r-1} t^\fl(1_{i-1},k,1_{r-i-1})
=-2\sum_{i=1}^{r-2} t^\fl(1_{i-1},k+1,1_{r-i-2}).
\end{align}
Repeat this for $k=2,3,\dots, n$, we see that
\begin{align*}
t_1^\fl(1_{r-1})
=-\sum_{i=1}^{n-1} t_1^\fl(1_{i-1},2,1_{r-i-1})
=-(-2)^{r-2} t^\fl(r)=2^{r-2} t^\fl(r)
\end{align*}
since $r$ is odd.
Finally, taking $k=1$ in \eqref{equ:t1rReduce}
yields the lemma immediately.
\end{proof}

\begin{lem}\label{lem:Un=-zeta(n)}
    For all $n\in\N$, $U_n=-\zeta(n)$ if $n\ge 2$.
\end{lem}
\begin{proof}
If $n=2$ then
\begin{align*}
U_2=&\, \int_0^1 \om_{-1}( \om_1-\om_{-1})
=\int_0^1 (\tb-\tc)2\tc.
\end{align*}
By the double stuffle relation,
\begin{align*}
\int_0^1 2\tc\tc=\zeta(\ol{1})^2=  2\zeta(\ol{1},\ol{1})+\zeta(2)
=\int_0^1 2\tb\tc +\zeta(2).
\end{align*}
Hence, $U_2=-\zeta(2)$.

In general, we see that by duality relation
\begin{align*}
U_n=&\, \int_0^1 \om_{-1} \om_1^{n-1}-\om_{-1} \om_1^{n-2}\om_{-1}\\
=&\,\int_0^1 (\om_1- \om_{-1}) (\om_0+ \om_1-\om_{-1})^{n-2}\om_0
=\zeta^\sharp(1_{n-2}, \ol{2})=-\zeta(n)
\end{align*}
by \eqref{equ:1nBar2}, as desired.
\end{proof}

\begin{lem}\label{lem:Un=-zetam(n)}
For all $n\in\N$, define
\begin{align*}
U_n^\fm=&\, \sum_{\eta_j=\pm 1} (\eta_1-\eta_1\eta_n)  I^\fm(0;\eta_1,\dots,\eta_n;1).
\end{align*}
Then $U_n^\fm=-\zeta^{\fm}(n)$.
\end{lem}
\begin{proof}
In view of Lemma~\ref{lem:Un=-zeta(n)} we only need to show that $D_rU_n^\fm=0$ for all odd $r<n.$ First,
\begin{align*}
   U_1=t^\fm(1) -T^\fm(1)=\zeta^{\fm}(1).
\end{align*}
Assume $n\ge 2$ and $U_k^\fm=-\zeta^{\fm}(k)$ for all $1\le k<n$. For any odd $r<n$, the coaction
\begin{align*}
D_r U_n^\fm=&\, \sum_{\eta_j=\pm 1} (\eta_1-\eta_1\eta_n) D_r I^\fm(0;\eta_1,\dots,\eta_n;1) \\
=&\, \sum_{\eta_j=\pm 1} (\eta_1-\eta_1\eta_n)
\Big[ I^\fl(0;\eta_1,\dots,\eta_r;\eta_{r+1})\ot I^\fm(0;\eta_{r+1},\dots,\eta_n;1) \\
&\,\qquad\quad  + I^\fl(\eta_{n-r};\eta_{n-r+1},\dots,\eta_n;1)\ot I^\fm(0;\eta_1,\dots,\eta_{n-r};1) \Big]
\end{align*}
since all the other cuts vanish by Lemma~\ref{lem:D1S1bfl}\ref{Murakami-lemma9}. Hence,
\begin{align*}
D_r U_n^\fm =&\, \sum_{\eta_j=\pm 1}
\Big[\eta_1 I^\fl(0;\eta_1,\dots,\eta_r;1)\ot \eta_{r+1} I^\fm(0;\eta_{r+1},\dots,\eta_n;1) \\
&\,\qquad\quad  -\eta_1 I^\fl(0;\eta_1,\dots,\eta_r;1)\ot \eta_{r+1} \eta_n I^\fm(0;\eta_{r+1},\dots,\eta_n;1)\\
&\,\qquad\quad  + \eta_1 I^\fl(0;\eta_{n-r+1},\dots,\eta_n;1)\ot I^\fm(0;\eta_1,\dots,\eta_{n-r};1)\\
&\,\qquad\quad   +\eta_1 I^\fl(\eta_{n-r};\eta_{n-r+1},\dots,\eta_n;0)\ot I^\fm(0;\eta_1,\dots,\eta_{n-r};1)\\
&\,\qquad\quad -\eta_1\eta_n I^\fl(0;\eta_{n-r+1},\dots,\eta_n;1)\ot I^\fm(0;\eta_1,\dots,\eta_{n-r};1)\\
&\,\qquad\quad  -\eta_1\eta_n I^\fl(\eta_{n-r};\eta_{n-r+1},\dots,\eta_n;0)\ot I^\fm(0;\eta_1,\dots,\eta_{n-r};1)\Big] =0
\end{align*}
since the first and the fifth terms cancel with each other directly while the second and the last term do the same. The
remaining third and fourth terms cancel each other by path reversal. The lemma now follows from Lemma~\ref{lem:Un=-zeta(n)}
immediately.
\end{proof}

\begin{thm}\label{thm:assumConjMotivic}
For all even positive integers $d$ the motivic MMV $V_0^\fm(d)$ is unramified. Moreover,
assuming Conjecture~\ref{conj:EulerSharp212bar}, we have
\begin{equation*}
V_0^\fm(d)= \frac12\zeta^{\fm,\star}(1,d)- \frac{1}4 \zeta^\fm(d+1)+\sum_{m+n=d+1,2|n} \zeta^\fm(m,n).
\end{equation*}
\end{thm}
\begin{proof}
For any odd positive odd integer $r<d$ we have
\begin{align*}
D_r V^\fm_0(d)=&\, \sum_{\eta_j=\pm 1} \eta_1\eta_d  D_r I^\fm(0;\eta_1,\eta_2,\dots,\eta_d,0;1)\\
=&\, \sum_{\eta_j=\pm 1} \eta_1\eta_d \Big[
I^\fl(0;\eta_1,\dots,\eta_r;\eta_{r+1})\ot  I^\fm(0;\eta_{r+1},\dots,\eta_d,0;1)\\
&\, \qquad  + I^\fl(\eta_{d-r};\eta_{d-r+1},\dots,\eta_d;0)\ot  I^\fm(0;\eta_1,\dots,\eta_{d-r},0;1) \\
&\, \qquad  +\delta_{r>1}I^\fl(\eta_{d-r+1};\eta_{d-r+2},\dots\eta_d,0;1)\ot  I^\fm(0;\eta_1,\dots,\eta_{d-r+1};1)
\Big]
\end{align*}
because all the other cuts vanish by Lemma~\ref{lem:D1S1bfl}\ref{Murakami-lemma9}. By changing the index $\eta_j\to
\eta_j\eta_{r+1}$ for $1\le j\le r$ in the first term and similar operations, we obtain
\begin{align*}
D_r V^\fm_0(d)=&\, \sum_{\eta_j=\pm 1} \Big[
\eta_1 I^\fl(0;\eta_1,\dots,\eta_r;1)\ot \eta_{r+1} \eta_d I^\fm(0;\eta_{r+1},\dots,\eta_d,0;1)\\
&\, \qquad  \quad +\eta_d I^\fl(1;\eta_{d-r+1},\dots,\eta_d;0)\ot  \eta_1\eta_{d-r} I^\fm(0;\eta_1,\dots,\eta_{d-r},0;1) \\
&\, \qquad +\delta_{r>1}\eta_1\eta_dI^\fl(0;\eta_{d-r+2},\dots\eta_d,0;1)\ot  I^\fm(0;\eta_1,\dots,\eta_{d-r+1};1)\\
&\, \qquad +\delta_{r>1}\eta_1\eta_dI^\fl(\eta_{d-r+1};\eta_{d-r+2},\dots\eta_d,0;0)\ot  I^\fm(0;\eta_1,\dots,\eta_{d-r+1};1)
\Big].
\end{align*}
For the second term, by changing the index $\eta_j\to \eta_{d-j+1}$ for $d-r+1\le j\le d$ and $\eta_j\to \eta_{j+r}$ for $1\le
j\le d-r$ we see that it cancels with the first term. Thus, we obtain
\begin{align*}
D_r V^\fm_0(d)=&\,\delta_{r>1} \sum_{\eta_j=\pm 1} \Big[
 \eta_dI^\fl(0;\eta_{d-r+2},\dots\eta_d,0;1)\ot \eta_1 I^\fm(0;\eta_1,\dots,\eta_{d-r+1};1)\\
&\, \qquad  \quad -\eta_d I^\fl(0;0,\eta_d,\dots\eta_{d-r+2};1)\ot  \eta_1\eta_{d-r+1} I^\fm(0;\eta_1,\dots,\eta_{d-r+1};1)
\Big].
\end{align*}
The path composition property immediately implies that
\begin{equation*}
I^\fl(0;\eta_{d-r+2},\dots\eta_d,0;1)=-I^\fl(1;\eta_{d-r+2},\dots\eta_d,0;0)
=I^\fl(0;0,\eta_d,\dots\eta_{d-r+2};1).
\end{equation*}
Hence, by setting $U_n^\fm=\sum_{\eta_j=\pm 1} (
 \eta_1-\eta_1\eta_n) I^\fm(0;\eta_1,\dots,\eta_n;1)$ we get
\begin{align*}
D_r V^\fm_0(d)=&\, \delta_{r>1} t_1^\fl(1_{r-1}) \ot U^\fm_{d-r+1}\in \calL_r\ot\calH_{d-r+1}
\end{align*}
from Lemma~\ref{lem:t111Unram} and Lemma~\ref{lem:Un=-zetam(n)}. This completes the proof of the first part of the theorem,
namely, $V^\fm_0(d)$ is unramified.

Turning to the exact expression of $V^\fm_0(d)$, due to Theorem~\ref{thm:assumConj}, we only need to show the two sides have the
same image under all the coactions.

From \cite[Prop. 4.1]{XuZhao2023Aug}, we can easily get
\begin{align*}
\zeta^\star(1,d)=&\, \frac{d+2}2 \zeta(d+1)-\frac12 \sum_{2<r<d,2\nmid r} \zeta^\fm(r)\zeta^\fm(d+1-r),\\
\sum_{m+n=d+1,2|n} \zeta^\fm (m,n)
=&\,\frac{d+2}4 \zeta^\fm(d+1)-\zeta^{\fm,\star}(1,d)
+\sum_{1\le r<d, \ r \text{ odd}} 2^{r-1} \zeta^\fm(r)\zeta^\fm(d+1-r).
\end{align*}
Hence, for all odd positive integer $r<d$  we have
\begin{align*}
&\, D_r\bigg(\frac12\zeta^{\fm,\star}(1,d)- \frac{1}4 \zeta^\fm(d+1)+\sum_{m+n=d+1,2|n} \zeta^\fm(m,n)\bigg) \\
=&\, D_r\bigg(\frac12\zeta^{\fm,\star}(1,d)
+\sum_{1\le r<d, \ r \text{ odd}} 2^{r-1} \zeta^\fm(r)\zeta^\fm(d+1-r) \bigg) \\
=&\, -\bigg(\frac12-2^{r-1}\bigg)  \zeta^\fl(r)\ot \zeta^\fm(d+1-r)\\
=&\, \delta_{r>1}t_1^\fl(1_{r-1}) \ot U^\fm_{d+1-r} =D_r V^\fm_0(d)
\end{align*}
by Lemmas \ref{lem:t111Unram} and~\ref{lem:Un=-zetam(n)}.
This completes the proof of the theorem by Theorem~\ref{thm:Glanois}.
\end{proof}

\begin{exa}
If we apply the duality relation to $V_0(d)$ we get an expression that is closely related to the Euler $\sharp$-sums defined
in \cite{Glanois2015} for which the descent results in \cite{LinebargerZh2015} could be used to evaluate $V_0(d)$ in terms of
(products of) Riemann zeta values precisely. For example, we have
\begin{equation*}
4V_0(8)=64 M(1_7,2; \{-1\}_7,1)=511\zeta(9)-254\zeta(2)\zeta(7)- 62 \zeta(4)\zeta(5) -14\zeta(6)\zeta(3).
\end{equation*}
\end{exa}

\section{Some general conjectures}\label{sec:Conj}
In Section~\ref{sec:UnramDepth3}, we have determined all depth three unramified motivic MMVs. To conveniently describe our
findings for large depth, we call the largest component of an index $\bfs\in\N^d$ its \emph{width}, denoted by $\wth(\bfs)$.
The following fact is useful when we study MZVs and MTVs both of which satisfy the duality relation.

\begin{lem}\label{lem:girth}
Let $\bfs\in\N$ and let $\bfs^*$ denote its dual. Then $\wth(\bfs^*)\le \dep(\bfs)+1$. Moreover, the equality holds if and only if the largest component of $\bfs^*$ is its only non-unit component.
\end{lem}
\begin{proof}
By the definition, $\wth(\bfs)+\dep(\bfs)-1\le |\bfs|$  and the equality holds if and only if all the non-maximal components are equal to 1. Thus,
\begin{equation*}
   \wth(\bfs)+\dep(\bfs)-1\le \dep(\bfs)+\dep(\bfs^*)
\end{equation*}
which yields the lemma immediately.
\end{proof}

We say $\bfs$ is \emph{wide} (resp. \emph{slender}, resp.  \emph{narrow}) if $\wth(\bfs)\ge 4$ (resp. $\wth(\bfs)=3$,  resp.
$\wth(\bfs)\le 2$.)

We start with a conjecture concerning the MTVs. Recall that Kaneko and Tsumura \cite{KanekoTs2020} conjectured that if an MTV $T(\bfs)$ is unramified then either $\dep(\bfs)\le 3$ or $\dep(\bfs^*)$. This motivates us the following using Lemma~\ref{lem:girth}.

\begin{conj}\label{conj:unramMTVdual}
A wide MTV of depth at least 4 is ramified.
A slender MTV is unramified if and only if it has the shape $T(1_{2k},3,1_{2l},2)$ or $T(1_{2k},3)$ for some
$k,l\in\N\cup\{0\}$.
A narrow MTV is unramified if and only if it has the shape $T(1_k,2)$ or $T(2,1_{2k+1},2)$ for some $k\in\N\cup\{0\}$, or
$T(2,1_k,2,1_l)$ for some $k,l\in\N\cup\{0\}$ with $k+l$ odd,  or
$T(1_{2k+1},2,1_{2l},2,1_{2m+1})$ for some $k,l,m\in\N\cup\{0\}$.
\end{conj}

\begin{rem}
By duality, $T(1_{2k},3,1_{2l},2)=T(2l+2,1,2k+2)$, $T(1_{2k},3)=T(1,2k+2)$,
$T(1_k,2)=T(k+2)$, $T(2,1_{2k+1},2)=T(2k+3,2)$, and $T(1_{2k+1},2,1_{2l},2,1_{2m+1},2)=T(2m+3,2l+2,2k+3)$ are all unramified
by \cite[Theorem~1.2]{XuZhao2023Aug}. This conjecture implies that if the depths of an MTV and its dual are both at least 4 then
it is ramified, which was first conjectured by Kaneko and Tsumura \cite{KanekoTs2020}.
\end{rem}

From numerical evidence, we propose the following conjectures.

\begin{conj}\label{conj:depth4MMVs}
In depth $4$, the following list should exhaust all wide or slender unramified MMVs that are not MZVs nor MtVs:
\begin{alignat*}{5}
&M(1, 2, 4, 2;-1, 1, 1, 1),\quad &
&M(4, 1, 2, 2;1, -1, 1, 1),\quad &
&M(3, 1, 1, 3;-1, 1, 1, -1),\\
&M(1, 4, 1, 2;1, 1, -1, 1),\quad &
&M(3, 2, 1, 2;1, 1, -1, 1),\quad &
&M(2, 2, 1, 4;-1, -1, 1, -1),\\
&M(2, 3, 1, 2;1, 1, -1, 1),\quad &
&M(1, 2, 1, 4;1, 1, -1, 1),\quad &
&M(4, 1, 2, 2;-1, 1, -1, -1),\\
&M(4, 2, 1, 2;1, 1, -1, 1),\quad &
&M(3, 3, 1, 2;1, 1, -1, 1),\quad &
&M(2, 4, 1, 2;1, 1, -1, 1),\\
&M(2, 2, 1, 4;1, 1, -1, 1),\quad &
&M(3, 1, 2, 2;1, -1, 1, 1),\quad &
&T(3, 1, 1, 2), \quad T(1, 1, 3, 2).
\end{alignat*}
Moreover, $T(1, 1, 1, 2)$ and $M(1, 2, 1, 2;1, 1, -1, 1)$ are the only \emph{narrow} unramified non-MZV and non-MtV elements
of depth $4$ besides the three unramified families covered in Theorems~\ref{thm:UnramMMVfamily1}, \ref{thm:UnramMMVfamily2} and
\ref{thm:UnramMMVfamily3}.
\end{conj}

The unramified non-MTV MMVs in Conjecture~\ref{conj:depth4MMVs} all have the following property: if there is only one odd
signature then it must correspond to a unit component, otherwise all the non-unit components must have odd signature while
all the unit components have even signature.

In \cite{XuZhao2023Aug}, we found that $S(2,1_3,4)$ is unramified. We believe this is truly special.
\begin{conj}\label{conj:depth5MMVs}
In depth $5$, $S(2,1_3,4)$ is only unramified \emph{wide} MMV that is not an MZV nor an MtV. An unramified slender or narrow
MMVs
must be either MZV, or MtV, or in the family of Theorems~\ref{thm:UnramMMVfamily1} or \ref{thm:UnramMMVfamily2}, or in the
following list:
\begin{alignat*}{7}
&T(1, 1, 1, 1, 2),\quad &
&T(2, 1, 1, 1, 2),\quad &
&M(2, 1, 2, 1, 2;1_3, -1, 1),\quad&
&M(1, 2, 1, 2, 2;1_2, -1, 1_2),\\
&T(1, 2, 1, 1, 2),\quad &
&T(1, 1, 2, 1, 2),\quad &
&M(1, 2, 2, 1, 2;1_3, -1, 1),\quad&
&M(1, 2, 1, 3, 2;-1, 1_4),\\
&T(1, 1, 1, 2, 2),\quad &
&T(1, 1, 1, 1, 3),\quad &
&M(3, 1, 2, 1, 2;1_3, -1, 1).
\end{alignat*}
\end{conj}

Again, in Conjecture~\ref{conj:depth5MMVs} all non-MTV MMVs have only one odd signature and it corresponds to a unit
component.

\begin{conj}\label{conj:depth>5MMVs}
Suppose that the depth is at least $6$. Then we have
\begin{enumerate}[label=\upshape{(\arabic*)},leftmargin=2cm]
    \item All unramified \emph{wide} MMVs are MZVs or MtVs, i.e., their parity patterns must be $(1,\dots,1)$ or
        $(-1,\dots,-1)$.

\item  All unramified \emph{slender} MMVs are MZVs or MtVs or the MTVs given in Conjecture~\ref{conj:unramMTVdual}.

\item  All unramified \emph{narrow} MMVs are MZVs, or MtVs, or the MTVs given in Conjecture~\ref{conj:unramMTVdual}, or
    elements in the families given in Theorems~\ref{thm:UnramMMVfamily1}, \ref{thm:UnramMMVfamily2} and
    \ref{thm:UnramMMVfamily3}.
\end{enumerate}
\end{conj}

In \cite{XuYanZhao2022Aug}, we initiated the study of alternating MMVs and explored their important properties such as their
regularization and double shuffle relations. By utilizing and extending the ideas of this paper, we have uncovered four families of unramified motivic alternating MMVs. Furthermore, we wonder if similar phenomena appear in other
levels. For example, in level three Borwein et al. defined real multiple Clausen values and multiple Glaishers values, and in
level five Broadhurst investigated the multiple Deligne values. It would be very interesting to determine if there are
unramified families among them. Presumably, the motivic theory developed in \cite{Glanois2015} should suffice to deal with the
level three problem, however, level 5 is still mostly an uncharted territory and any progress along this line would provide us
more insight into the theory of motives over the general cyclotomic fields.

\medskip
\noindent
\textbf{Funding.} Ce Xu is supported by the General Program of Natural Science Foundation of Anhui Province (Grant No. 2508085MA014). J. Zhao is supported by the Jacobs Prize from the Bishop's School.

\medskip
\noindent
\textbf{Financial interest.} The authors have no relevant financial or non-financial interests to disclose. 

\medskip
\noindent
\textbf{Competing interest.} The authors have no competing interests to declare that are relevant to the content of this article.

\end{document}